\documentclass[reqno,oneside]{amsart}
\usepackage{amssymb,mathrsfs,amsmath,amsthm}
\usepackage{xcolor}
\usepackage{tikz}
\usepackage{tikz-cd}
\usetikzlibrary{calc,positioning}
\usepackage[inline,shortlabels]{enumitem}
  \setlist{topsep=2pt, itemsep=3pt,partopsep=3pt}
\usepackage[a4paper, margin=1in]{geometry}
\usepackage[colorlinks,citecolor=brown,linkcolor=violet,urlcolor=teal]{hyperref}

\newtheorem{theorem}{Theorem}[section]
\newtheorem{lemma}[theorem]{Lemma}

\newtheorem{proposition}[theorem]{Proposition}

\newtheorem{question}[theorem]{Question}

\newtheorem{corollary}[theorem]{Corollary}

\theoremstyle{definition}
\newtheorem{definition}[theorem]{Definition}
\newtheorem{example}[theorem]{Example}

\newtheorem{remark}[theorem]{Remark}
\theoremstyle{plain}

\numberwithin{equation}{section}

\newcommand\Supp{\operatorname{Supp}}
\newcommand\Spec{\operatorname{Spec}}

\newcommand\coeff{\operatorname{coeff}}

\title[Geometric singularities of regular surfaces with nef anti-canonical divisors]%
      {Geometric singularities of regular surfaces with nef anti-canonical divisors over imperfect fields}
\author{Chongning Wang}
\email{chongningwang@hbmzu.edu.cn}
\address{School of Mathematics and Statistics, Hubei Minzu University, Enshi 445000, P.R.China.}

\author{Lei Zhang}
\email{zhlei18@ustc.edu.cn}
\address{School of Mathematical Sciences, University of Science and Technology of China, Hefei 230026, P.R.China.}

\begin{document}
\begin{abstract}
Let $S$ be a regular projective surface over a field $k$ of characteristic $p>0$, with $H^0(S,\mathcal{O}_S)=k$ and $-K_S$ nef.  We prove that $S$ is geometrically integral over $k$ when $p\geq 7$, and we also find an example of $S$ that is not geometrically integral when $p=5$.
\end{abstract}

\maketitle
\section{Introduction}\label{section-intro}
Let $X$ be a regular projective variety over a field $k$ of characteristic $p>0$.
When $k$ is imperfect, $X$ is not necessarily smooth.
Such varieties arise naturally as generic fibers of fibrations $f\colon\mathcal X\to T$ between smooth varieties with $k=K(T)$.
The singularities that $X$ acquires after base change to $\overline k$---its \emph{geometric singularities}---can be quite bad.
For instance, over $k=\mathbb F_p(s_1,\cdots, s_n)$, the hypersurface defined by the equation $X_0^p+s_1X_1^p+\cdots + s_nX_n^p=0$ in $\mathbb{P}_k^n$ is regular, while its base change to 
$\overline k$ is a $p$-fold hyperplane.
In this paper, we focus on varieties with $-K_X$ being nef, such as the generic fibers of Mori fiber spaces or of Iitaka fibrations of $K_X$.
Such varieties play an important role in the classification of varieties in characteristic $p$, and understanding their geometric singularities is a fundamental problem.
It is expected that they have mild geometric singularities when $p$ is large; for example, the anticanonical divisor of the above hypersurface is nef if and only if $p\leq n+1$.
The following cases have been extensively studied:
\begin{itemize}
  \item $\dim X = 1$ (\cite{Queen-71-I,schroer_fibration_2010,tanaka_invariants_2021,Schroer2211}):
    if $-K_X$ is ample, then $X$ is a conic in $\mathbb{P}^2$, which is smooth when $p\geq 3$;
    and if $K_X \equiv 0$, then $X$ is smooth when $p\geq 5$, while for $p=2,3$ such curves are completely classified \cite{Schroer2211}.
  \item $\dim X = 2$ with $-K_X$ ample (\cite{Mori-Saito-03Wild, Patakfalvi-Waldron-2022, Fanelli-Schroer-2020, Bernasconi-Tanaka-22, Bernasconi-Martin-2024-Bounding, BT2408} among others): $X$ is geometrically normal when $p\ge5$ (\cite[Theorem~1.5]{Patakfalvi-Waldron-2022} or \cite[Theorem~3.7]{Bernasconi-Tanaka-22}), and geometrically regular when $p\ge11$ (\cite[Proposition~5.2]{Bernasconi-Tanaka-22}).
\end{itemize}

\begin{question}\label{question-ques}
  Let $S$ be a regular projective surface over $k$ with $-K_S$ nef and $H^0(S,\mathcal{O}_S)=k$.
  \begin{itemize}
    \item[\rm (Q1)] Is there a number $N_1$ such that every such $S$ in characteristic $p>N_1$ is geometrically integral?
    \item[\rm (Q2)] Is there a number $N_2$ such that every such $S$ in characteristic $p > N_2$ is geometrically normal?
  \end{itemize}
\end{question}

When $p=2$ or $3$, it is known that $S$ can fail to be geometrically integral.
Let $k_0$ be a perfect field of characteristic $p$ and let $k=k_0(a_0,\ldots,a_9)$ with the $a_i$ algebraically independent.
The following surfaces are regular and geometrically non-reduced:
\begin{align*}
  -K_S\ \text{ample}: &\quad \textstyle\sum_{i=0}^{3}a_iX_i^{p}=0\ \subset\ \mathbb P^3_k , \\
  K_S\sim0,\ p=2: &\quad \textstyle\sum_{i=0}^{3}a_iX_i^{4}=0\ \subset\ \mathbb P^3_k , \\
  K_S\sim0,\ p=3: &\quad \textstyle\sum_{i=0}^{4}a_iX_i^{3}=\sum_{i=0}^{4}a_{i+5}X_i^{2}=0\ \subset\ \mathbb P^4_k .
\end{align*}
We therefore mainly consider characteristic $p\ge5$.

\subsection{The main results}\label{subsection-main-results}%
Our first result answers (Q1) affirmatively.

\begin{theorem}[regular case]\label{theorem-main-regular-case}%
  Let $k$ be a field of characteristic $p$.
  Let $S$ be a regular projective surface over $k$ with $H^0(S,\mathcal{O}_S)=k$ and $-K_S$ nef.
  If $p\ge7$, then $S$ is geometrically integral over $k$.
\end{theorem}

More generally, we prove the following result for log canonical pairs; the same argument yields a precise description of the case $p=5$.
\begin{theorem}[general case, $=$ Theorem~\ref{theorem-geo-int-log-pair}]\label{theorem-main-pair}
  Let $k$ be a field of characteristic $p\ge 5$.
  Let $S$ be a projective normal surface over $k$ with $H^0(S,\mathcal{O}_S)=k$, and let $\Delta$ be an effective $\mathbb Q$-divisor on $S$ such that
  \begin{itemize}
    \item $-(K_S + \Delta)$ is nef, and
    \item the pair $(S,\Delta)$ is log canonical.
  \end{itemize}
  Then the following hold.
  \begin{enumerate}[\rm(1)]
    \item If $p\ge7$, then $S$ is geometrically reduced.
    \item If $p=5$ and $S$ is geometrically non-reduced, then $\Delta=0$, $K_S\equiv 0$, and $S$ has at most canonical singularities.
      Moreover, writing $\widetilde S\to S$ for the minimal resolution, the surface $X:=(\widetilde S\otimes_{k} k^{1/p})_{\rm red}^{\nu}$ has a unique non-canonical singular point $P$, a rational singularity of multiplicity $3$ whose minimal resolution has exceptional locus $E_1 \cup E_2$ with $(E_1)^2_{k^{1/p}} = -3$ and $(E_2)^2_{k^{1/p}} = -2$.
  \end{enumerate}
\end{theorem}

The bound $p\geq 7$ in Theorem~\ref{theorem-main-pair}\,(1) is optimal.
Guided by the description in Theorem~\ref{theorem-main-pair}\,(2), and with the help of an AI system, we construct a counterexample in characteristic $5$; see \S\ref{subsection-exam-K-trivial-5}.

\begin{theorem}[$=$ Theorem~\ref{theorem-exam-K-trivial-5}]\label{theorem-main-counterexample}
  Let $k:=\mathbb F_5(s_1,s_2)$, which has characteristic $5$ and degree of imperfection~$2$.
  Then there exists a projective \emph{regular} surface $S$ over $k$ with $H^0(S,\mathcal O_S)=k$ and $K_S\equiv0$ such that $S$ is geometrically non-reduced.
\end{theorem}

\begin{remark}\label{remark-intro}
  \begin{enumerate}[\rm(1)]
    \item Without the log canonical hypothesis, Theorem~\ref{theorem-main-pair} fails in every characteristic (Example~\ref{example-pair}).
    \item Concerning (Q2), we exhibit for $p=5$ and $p=7$ regular, geometrically integral, geometrically non-normal surfaces $S$ with $K_S \equiv 0$ (Example~\ref{example-57}); so $N_2\ge7$ if it exists at all.
    \item We show in Proposition~\ref{proposition-MF-geo-normal} that, if $S$ admits a smooth Mori fiber space structure over a curve, then $S$ is geometrically regular for all $p\ge5$.
  \end{enumerate}
\end{remark}

\subsection{Applications}%
Let $X$ be a smooth projective variety over an algebraically closed field of characteristic $p>0$ with $-K_X$ nef.
By \cite[Theorem~1.3]{EP23}, the Albanese morphism $a_X\colon X\to A$ is surjective; if $X\buildrel f\over\to Y \buildrel g\over\to A$ is its Stein factorization, then $Y\to A$ is purely inseparable, and if moreover $f$ is separable---equivalently, the generic fiber $X_{\eta}$ is geometrically integral---then $Y=A$, i.e., $a_X$ is a fibration.
When $a_X$ has relative dimension one, it is always a fibration by \cite[Theorem~1.5]{CWZ26a}.
Combining these with Theorem~\ref{theorem-main-regular-case} gives the following.
\begin{corollary}[special case of Corollary~\ref{corollary-fib-codim2}]\label{corollary-intro-fib-codim2}
  Let $k$ be a perfect field of characteristic $p\ge7$.
  Let $X$ be a smooth projective variety over $k$ with $-K_X$ nef, and denote by $a_X\colon X\to A$  the Albanese morphism.
  If $\dim X -\dim a_X(X) \leq 2$, then $a_X$ is a fibration.
\end{corollary}

\begin{remark}
   In the companion paper \cite{CWZ26b}, joint with J.~Chen, we establish the effectivity of canonical divisors on regular $K$-trivial surfaces over arbitrary fields.
  A key ingredient there is Theorem~\ref{theorem-main-regular-case}: when $p\ge7$, it is used to show that a regular, geometrically non-normal $K$-trivial surface $S$ satisfies $K_S\sim0$.
\end{remark}

\subsection{Strategy of the proof of Theorem~\ref{theorem-main-pair}}\label{subsection-strategy}
Passing to a minimal resolution, we may assume that $S$ is regular. 
Since $-K_S\equiv \Delta-(K_S+\Delta)$ is pseudo-effective, either $K_S$ is not pseudo-effective and a $K_S$-MMP terminates with a Mori fiber space $Z\to B$, or $K_S\equiv0$.
It is not difficult to treat the first case, precisely, it is either a direct consequence of the result of del Pezzo surface or can be proved by similar argument from \cite{Patakfalvi-Waldron-2022} or \cite{Bernasconi-Tanaka-22}; see \S\ref{section-MF}.
The really difficult case from our point of view is $K_S \equiv 0$ which is treated in \S\ref{section-proof}. Let us explain our strategy and the key ingredients to treat this case as follows. We can assume that $S$ is geometrically non-reduced, then $S_{k^{1/p}}$ is non-reduced. Set $X:=(S_{k^{1/p}})_{\mathrm{red}}^{\nu}$ and consider the natural morphism $\pi\colon X \to S$. By the formula of Ji--Waldron \cite[Theorem~1.1]{ji_structure_2021}:
\begin{equation}\tag{C1}\label{eq:C1}
  K_X + (p-1) \mathfrak{C} \equiv 0,
  \qquad \mathfrak{C}=\mathfrak{M} + \mathfrak{F},\quad
  \mathfrak{M}>0\ \text{movable},\ \ \mathfrak{F}\geq 0.
\end{equation}
We stop after this single Frobenius base change because geometric non-reducedness is already visible at $k^{1/p}$; but iterating up to $\overline k$ would lose the numerical information we need. More precisely, regularity of $S$ yields
\begin{equation}\tag{C2}\label{eq:C2}
  pC\ \text{is Cartier for every integral curve } C\subset X,
\end{equation}
and thus $C\cdot_{k_C} D \in \tfrac1p\mathbb{Z}$ for every Weil divisor $D$, where $k_C := H^0(C^\nu, \mathcal{O}_{C^\nu})$; we refer to this as $p$-factoriality.
Take a minimal resolution $\sigma\colon \widetilde{X}\to X$ and run a $K_{\widetilde{X}}$-MMP, ending with a Mori fiber space $Z\to B$: \[
  \begin{tikzcd}[row sep=tiny]
    &\widetilde{X} \ar[ld,"\sigma"'] \ar[rd,"\epsilon"]\ar[ddddr,"f"'] \\
    X && Z \ar[ddd,"g"] \\ \\ \\
      && B \rlap{\,.}
  \end{tikzcd}
\]
Then $K_{\widetilde{X}} + (p-1) (\widetilde{\mathfrak{M}} + \widetilde{\mathfrak{F}}) + \sum_i m_i E_i \equiv 0$ with the $E_i$ the $\sigma$-exceptional divisors, and pushing down to $Z$ gives $K_Z + (p-1)(M_Z +F_Z) + E_Z \equiv 0$.
For $\dim B = 0$ we invoke boundedness of regular del Pezzo surfaces \cite[Theorem~1.8]{Tanaka-24-Bound}.
For $\dim B = 1$, since $p-1>2$, one deduces that $\widetilde{\mathfrak{M}} + \widetilde{\mathfrak{F}}$ is $g$-vertical and that exactly one component of $\sum_i m_i E_i$, say $E_1$, is horizontal over $B$; the contradiction then comes from the precise configuration of the $E_i$ and the fibers $F_b$.
Unlike the del Pezzo case, $\widetilde X$ may genuinely contain $(-1)$-curves, as the example of Theorem~\ref{theorem-main-counterexample} shows, so the MMP is a real birational modification and the situation is much more involved.
The following observations, which are consequences of $p$-factoriality, are what make the configuration tractable:
\begin{itemize}
  \item a general member $M \in |\mathfrak{M}|$ contains a component $\overline{F}_b:=\sigma_* F_b$ for some closed fiber $F_b$ of $f$ over $b \in B$, and we have $(\overline{F}_b)^2_{k(b)} = 1/p$ or $2/p$.
  \item if $(\overline{F}_b)^2_{k(b)} = 1/p$, then for every closed fiber $F_t$ one has $(\overline{F}_t \cdot \overline{F}_b)_{k(t)} = 1/p$, which forces exactly one component $T$ of $F_t$, of multiplicity one, to pass through $P=\sigma(E_1)$.
\end{itemize}
Here we use the notation $(C_1 \cdot C_2)_{k'}$ for the intersection number counted over the field $k'$.

For $p=5$ this argument ceases to give a contradiction, and Theorem~\ref{theorem-main-counterexample} shows that the configuration $(\overline F_b)^2_{k(b)}=2/p$ is genuinely realized.

\begin{remark}%
  When $S$ is geometrically integral but geometrically non-normal, the linear system $|\mathfrak{C}|$ usually has no movable part, and its strict transform $\widetilde{\mathfrak{C}}$ may be contracted by $\epsilon$. To prove that $S$ is geometrically normal when $p$ is large, the above strategy alone seems not enough, even for $\dim B = 0$.
\end{remark}

\subsection*{Acknowledgements}
The authors are grateful to Fabio Bernasconi for his helpful comments on a previous version of this manuscript.

This research is partially supported by CAS Project for Young Scientists in Basic Research (No.~YSBR-032), NSFC (No.~12122116 and No.~12471495), and Quantum Science and Technology-National Science and Technology Major Project (QNMP, No.~2021ZD0302902). 
The first author is also supported by Hubei Minzu University under Grant No.~BS26011.

\subsection*{AI Statement}%
An agentic AI system, mainly Fable 5, was used to help construct the characteristic-$5$ example of Theorem~\ref{theorem-main-counterexample}.

\section{Preliminaries}\label{section-preliminary}
\subsection{Conventions} 
\begin{itemize}
  \item Throughout this paper, $k$ is a field of characteristic $p>0$.
   By a $k$-scheme we mean a separated scheme of finite type over $k$.
   By a \emph{variety} over $k$ we mean an integral $k$-scheme.  By a \emph{surface} over $k$ we mean a variety of dimension two.
 \item We write $k\hookrightarrow k^{1/p}$ for the Frobenius map of $k$, and if $X$ is a scheme over $k$, then we write $X_{k^{1/p}}$ for $X \times_k k^{1/p}$.
  \item For a scheme $X$, we denote by $X_{\rm red}$ the associated reduced scheme and by $X_{\rm red}^\nu$ its normalization.
  \item For a variety $X$ over $k$, we denote by $k_X$ the algebraic closure of $k$ in the function field $K(X)$. 
    Since $X$ is of finite type over $k$, $k_X$ is a finite extension of $k$.
  \item For a birational morphism $\sigma\colon X\to Y$ of varieties and a Weil divisor $D$ on $Y$, we write $\sigma^{-1}_*(D)$ for the strict transform of $D$ on $X$.
  \item By a \emph{fibration} $f\colon X\to B$ we mean a proper surjective morphism to a normal variety with $f_*\mathcal O_X = \mathcal O_B$.
  \item Let $X$ be a regular projective surface with $H^0(X, \mathcal{O}_X) = k$.
    We call $X$ a \textit{del Pezzo surface} if $-K_X$ is ample.
    We call a fibration $f \colon X \to B$ a \emph{Mori fiber space} if $\dim B< \dim X$, $\rho(X/B) = 1$, and $-K_X$ is relatively ample over $B$.
  \item By a \emph{pair} $(X,\Delta)$ we mean a normal variety $X$ and an effective $\mathbb Q$-divisor $\Delta$ such that $K_X + \Delta$ is $\mathbb Q$-Cartier.
    We refer to \cite[Definition~2.8]{Kollar-Sing} for the definitions of canonical and log canonical singularities.
\end{itemize}

\subsection{Field extension and singularities}\label{section-field-extension}%
Let $X$ be a variety over $k$.
\begin{enumerate}
  \item If $X$ is proper and normal, then $H^0(X, \mathcal{O}_X)=k_X$ (\cite[Proposition~2.1]{tanaka_invariants_2021}).
  \item The variety $X$ is geometrically irreducible if and only if $K(X) \cap k^s = k$, where $k^s$ is the separable closure of $k$ (\cite[Corollary~3.2.14]{Liu-AGAC}).
    In particular, if $k$ is algebraically closed in $K(X)$, then $X$ is geometrically irreducible.
  \item Combining (1) and (2), we see that if $X$ is a normal proper variety with $k_X=k$, then $X$ is geometrically integral if and only if it is geometrically reduced.
  \item If $X$ and $Y$ are varieties over $k$ birational to each other, then $X$ is geometrically reduced over $k$ if and only if $Y$ is geometrically reduced over $k$ \cite[Proposition~2.12]{tanaka_invariants_2021}.
  \item If $X$ is regular, then for any separable field extension $k \subset k'$ the base change $X \times_k k'$ is regular (\cite[Lemma~2.6]{tanaka_invariants_2021}).
   \item The variety $X$ is geometrically regular (resp.\ geometrically reduced, geometrically normal) if and only if $X \times_k k^{1/p}$ is regular (resp.\ reduced, normal) (\cite[Proposition~2.10]{tanaka_invariants_2021}).
\end{enumerate}

\subsection{Foliation theory}%
We refer to \cite{Rudakov-Shafarevich-1976,Ekedahl87F,MP97} for the basic theory of foliations.
Since the ground field $k$ is imperfect, we work with subsheaves of the absolute tangent sheaf $T_{X/k_0}$ rather than of the relative tangent sheaf $T_{X/k}$.

\begin{definition}%
  Let $k_0$ be a perfect field of characteristic $p$, let $k/k_0$ be a finitely generated field extension, and let $X$ be a normal variety over $k$.
  A \emph{foliation} on $X$ is an $\mathcal O_X$-submodule $\mathcal F\subseteq T_{X/k_0}$ that is saturated (i.e.\ $T_{X/k_0}/\mathcal F$ is torsion-free), involutive, and $p$-closed.
\end{definition}

For a foliation $\mathcal F$ on $X$, the sheaf $\mathcal O_X^{\mathcal F}$ of $\mathcal F$-invariant functions is defined by
\[
  \mathcal O_X^{\mathcal F}(U)=\bigl\{f\in\mathcal O_X(U)\bigm|\delta(f)=0\text{ for all }\delta\in\mathcal F(U)\bigr\}.
\]
The \emph{quotient} by $\mathcal F$ is the normal variety
\[
  X/\mathcal F:=\operatorname{Spec}_X\mathcal O_X^{\mathcal F},
\]
together with the finite purely inseparable morphism $\pi\colon X\to X/\mathcal F$ of degree $p^{\operatorname{rk}\mathcal F}$.
Since the derivations in $\mathcal F$ are only $k_0$-linear, $\mathcal O_X^{\mathcal F}$ need not be a sheaf of $k$-algebras.
Hence, the quotient $X/\mathcal F$ is naturally just a variety over the subfield
\[
  k^{\mathcal F}:=\bigl\{a\in k\bigm|\delta(a)=0\text{ for all local sections }\delta\text{ of }\mathcal F\bigr\}\subseteq k.
\]

\smallskip
When $X$ is regular, the quotient $X/\mathcal F$ is regular at $\pi(x)$ if and only if $\mathcal F$ is a subbundle of $T_{X/k_0}$ at $x$, i.e.\ the stalk $(T_{X/k_0}/\mathcal F)_x$ is a free $\mathcal O_{X,x}$-module \cite[Proposition~2.4]{Ekedahl87F}.
In particular, if $\mathcal F$ is a subbundle everywhere, then $X/\mathcal F$ is regular, and the canonical bundle formula of \cite[Corollary~3.4]{Ekedahl87F} reads
\begin{equation}\label{eq:foliation-1}
  \omega_X\cong\pi^*\omega_{X/\mathcal F}\otimes(\det\mathcal F)^{\otimes(p-1)}.
\end{equation}

More generally, let $X$ be a normal variety. 
Then $\det\mathcal F:=(\bigwedge^{\operatorname{rk}\mathcal F}\mathcal F)^{\vee\vee}$ is a rank-$1$ reflexive sheaf and hence determines a Weil divisor class $c_1(\det\mathcal F)$. Restricting to the big open subset on which $X$ is regular and $\mathcal F$ is a subbundle, and then extending Weil divisor classes across codimension at least two, the formula (\ref{eq:foliation-1}) yields
\begin{equation}\label{eq:foliation-2}
  K_X\sim\pi^*K_{X/\mathcal F}+(p-1)\,c_1(\det\mathcal F).
\end{equation}

The following example shows that a quotient by a foliation can have a smaller field of constants and be regular but geometrically nonreduced over that field.

\begin{example}\label{example-foliation-geo-nonred-curve}
  Let $k_0$ be a perfect field of characteristic $p>0$ and let $\ell=k_0(s,t)$, so that $\ell$ has $p$-degree $2$.
  Set $k=\ell(a,b)$ with $a^p=s$ and $b^p=t$.
  Let $X=\mathbb P^1_k$ with affine coordinate $u$, and let
  \[
    \delta=\partial_a-u\,\partial_b\in\operatorname{Der}_{k_0}\bigl(K(X)\bigr).
  \]
  Then $\delta^{[p]}=0$, and
  \[
    \mathcal F:=T_{X/k_0}\cap K(X)\cdot\delta
  \]
  is a saturated rank-$1$ foliation.
  On constants, $\delta(a)=1$ and $\delta(b)=-u$.
  For $f\in k=\ell(a,b)$, we have $\delta(f)=f_a-u f_b$.
  Thus $\delta(f)=0$ implies $f_a=f_b=0$, so $k^{\mathcal F}=\ell$.
  Let
  \[
    C:=\bigl\{x^p+sy^p+tz^p=0\bigr\}\subset\mathbb P^2_\ell.
  \]
  Its function field is $K(C)=\operatorname{Frac}\bigl(\ell[\xi,\eta]/(\xi^p+s\eta^p+t)\bigr)$, where $\xi=x/z$ and $\eta=y/z$ on $C\cap\{z\ne0\}$.
  The morphism
  \[
  \phi\colon X\to C, \qquad [U:V]\mapsto [-aU-bV:U:V]
  \]
  is finite purely inseparable of degree $p$.
  On the chart $X\cap\{V\ne0\}$, we have $\phi^*\xi=-au-b$ and $\phi^*\eta=u$.
  Hence $\phi^*K(C)=\ell(u,au+b)=\ker\delta$.
  Thus $\mathcal F=\ker(T_{X/k_0}\to\phi^*T_{C/k_0})$ and therefore $X/\mathcal F\simeq C$ as varieties over $\ell=k^{\mathcal F}$.
  Clearly, the curve $C$ is regular and geometrically nonreduced over $H^0(C,\mathcal O_C)=\ell$. 
\end{example}

\subsection{Intersection theory}\label{subsection-intersection}%
  Let $X$ be a proper variety over $k$.
  Let $D$ be a Cartier divisor on $X$, and $C$ an integral curve on $X$.
  We define a 0-cycle on $C$ as follows \[
    (D \cdot C)_{\rm cycle} := \sum \mathrm{ord}_{\mathfrak p_i}(D|_{C}) \mathfrak p_i = \sum n_i \mathfrak p_i .
  \]
  For the definition of $\mathrm{ord}_{\mathfrak p_i}(D|_{C})$, we refer the reader to \cite[\S1.2]{fulton_intersection_theory} or \cite[\S7.2.1]{Liu-AGAC};
  for instance, if $\mathfrak p\in C$ is a regular closed point and $r\in \mathcal{O}_{C,\mathfrak p}$ a local equation for $D$, then 
  \[
    \mathrm{ord}_\mathfrak p(D) = \max \{ n \mid r \in \mathfrak m^n_{C,\mathfrak p}\}.
  \]
\begin{definition}[Intersection number]\label{definition-intersection-number}
  With the above notation, we define
  \[
    (D\cdot C)_k := \deg_k(\sum n_i \mathfrak p_i) = \sum n_i [\kappa(\mathfrak p_i):k].
  \]
\end{definition}
  This definition extends to $\mathbb{Q}$- or $\mathbb{R}$-Cartier divisors $D$, and to $1$-cycles $C$,
  but in this subsection we keep $D$ Cartier and $C$ an integral curve.
  
\begin{remark} 
  The intersection number $(D\cdot C)_k$ coincides with $\deg_k (D|_{C})$ as defined in \cite[Definition~7.3.1]{Liu-AGAC}.
  By the Riemann--Roch theorem (cf. \cite[Theorem~7.3.17]{Liu-AGAC}), this number can be equivalently expressed as 
  \[
    (D\cdot C)_k = \chi_k (C, D|_C) - \chi_k(C,\mathcal{O}_C).
  \]
  Here for a proper variety $X$ over a field $k$ and a coherent sheaf $M$ on $X$, we use the convention
  $$\chi_k(X,M) := \sum_{i=0}^{\dim X} (-1)^i\dim_k H^i(X,M).$$
\end{remark}

\begin{proposition}
  With notation as above, let $\nu\colon C^\nu\to C$ be the normalization of $C$.
  Then 
    \[
      (D\cdot C)_k = \sum_{\mathfrak q} \mathrm{ord}_{\mathfrak q} (D|_{C^\nu}) [\kappa(\mathfrak q):k] ,
    \]
     where the sum runs over the closed points $\mathfrak q$ of $C^\nu$.
\end{proposition}
\begin{proof}
  See \cite[Example~1.2.3]{fulton_intersection_theory}.
\end{proof}

Sometimes, it is convenient to count over the field $\ell := H^0(C^\nu,\mathcal{O}_{C^\nu})$ and write 
  \[
    (D\cdot C)_\ell := \sum_{\mathfrak q}\operatorname{ord}_{\mathfrak q}(D|_{C^\nu})[\kappa(\mathfrak q):\ell] = \frac{1}{[\ell:k]} (D\cdot C)_k.
  \]
We will also denote this intersection number by $D\cdot_\ell C$.
For self-intersections, we use the notation $(C)^2_k := (C\cdot C)_k$.

\subsection{Adjunction formula}\label{subsection-adjunction}
\begin{proposition}[{\cite[Proposition~2.35]{Kollar-Sing}}]\label{proposition-adj}
  Let $X$ be a normal surface and $C\subset X$ a reduced curve, and let $C^\nu\to C$ be its normalization.
  Assume that $K_X + C$ is $\mathbb Q$-Cartier.
  \begin{enumerate}[\rm(1)]
    \item There exists an effective divisor $\Delta_{C^\nu}$ on $C^\nu$ such that \[
        (K_X + C)|_{C^\nu} \sim_{\mathbb{Q}} K_{C^\nu} + \Delta_{C^\nu}.
      \]
      More precisely, if $n(K_X + C)$ is Cartier for some positive integer $n$, then $n\Delta_{C^\nu}$ is a $\mathbb Z$-divisor and \[
        n(K_X + C)|_{C^\nu} \sim nK_{C^\nu} + n\Delta_{C^\nu}.
      \]
    \item $\Delta_{C^\nu} = 0$ if and only if $C$ is regular and $X$ is regular along $C$.
  \end{enumerate}
\end{proposition}

\begin{example}
 If $X$ is a normal projective surface that is regular along a regular integral curve $C \subset X$ such that $p_a(C) = 0$, then we have \[
  (K_X + C) \cdot_{k_C} C = \deg_{k_C} (K_C) = -2.
 \]
\end{example}

\subsection{Behavior of canonical divisors under Frobenius base change}\label{subsection-geo-non-reduced}
The original statement of the following result assumes that $k$ is $F$-finite, but this assumption can be dropped by a standard argument (see \cite[p.~3917]{Das-Waldron-2022}).
\begin{theorem}[{\cite[Theorem~1.1]{ji_structure_2021}}]\label{theorem-can-bc}
  Let $X$ be a normal projective variety over $k$ with $H^0(X,\mathcal{O}_X) = k$.
  Set $Y := (X \times_k k^{1/p})_{\rm red}^\nu$ and denote by $\pi\colon Y\to X$ the induced morphism.
  Then there exists an effective divisor $\mathfrak C$ with movable part $\mathfrak{M}$ and fixed part $\mathfrak{F}$ on $Y$ such that \[
    \pi^*K_X \sim K_{Y} + (p-1)(\mathfrak{M} + \mathfrak{F}).
  \]
  Furthermore, $\mathfrak C > 0$ if and only if $X$ is geometrically non-normal, and $\mathfrak M>0$ if and only if $X$ is geometrically non-reduced.
\end{theorem}

\subsection{Blow-up}\label{section-bl}
Let $X$ be a surface and let $x\in X$ be a regular closed point.
Denote by $\mathcal{I} \subseteq \mathcal{O}_X$ the ideal sheaf corresponding to the closed point $x\in X$.
The blow-up of $X$ at $x$ is \[
 \mu: \widetilde{X} = \mathrm{Bl}_x(X) := \mathrm{Proj} \bigoplus_{i\ge0} \mathcal{I}^i \to X .
\]
The exceptional curve $E\cong \mathrm{Proj} \bigoplus_i \mathfrak m^i_x / \mathfrak m_x^{i+1}$ satisfies $H^0(E,\mathcal{O}_E) = k(x)$, and the variety
$\widetilde{X}$ is regular along $E$ (cf.~\cite[Proposition~3.1]{Liu2025}).
For an effective Cartier divisor $D$ on $X$, we have
  \begin{equation}
    \mu^* D = \widetilde{D} + \mathrm{ord}_x(D) E,
  \end{equation}
  where $\widetilde{D}$ is the strict transform of $D$ (cf. \cite[Example~4.3.9]{fulton_intersection_theory}).
  We remind the reader that even when $X$ is smooth over $k$, the blow-up $\widetilde{X}$ need not be smooth over $k$ and may even be geometrically non-normal; see \cite[Remark~2.7]{BFSZ2404}.

\subsection{Contractions and exceptional curves}\label{subsection-contractions}
Let $(x\in X)$ be a germ of a normal surface singularity.
Let $\mu\colon \widetilde{X}\to X$ be the minimal resolution of the singularity, and let $E = \bigcup E_i$ be the reduced exceptional locus.
We say that $(x\in X)$ is a {\it rational singularity} if $\dim_k (R^1\mu_* \mathcal{O}_{\widetilde{X}})_x = 0$.

Recall from \cite[pp.~131--132]{Artin66Rat} that there is a unique nonzero cycle $Z =\sum r_i E_i$ ($r_i\in \mathbb Z_{\ge0}$) that is {\it minimal} among all nonzero cycles $Z'=\sum a_i E_i$ ($a_i\in \mathbb Z_{\ge 0}$) satisfying $(Z'\cdot E_i) \le 0$ for all $i$.
We call $Z$ the {\it fundamental cycle} of $\bigcup E_i$.

\begin{proposition}[Castelnuovo's contraction criterion, {\cite[Theorem~27.1]{Lipman69}}]\label{proposition-cont}
  Let $X$ be a normal projective surface over $k$.
  Let $E_1,\ldots,E_n$ be distinct integral curves on $X$ such that $\bigcup_i E_i$ is connected and the intersection matrix $((E_i\cdot_k E_j))$ is negative-definite.
  Let $Z$ be the fundamental cycle of $\bigcup_i E_i$.
  Then there exists $h\colon X\to Y$ contracting $\bigcup_i E_i$ rationally to a point $P$ if and only if $\chi(Z) > 0$.
  When this condition holds, $Y$ is regular at $P$ if and only if the \emph{multiplicity} $m_P := - (Z)^2_k/ h^0(Z,\mathcal{O}_Z)$ of $P$ on $Y$ is $1$.
\end{proposition}

\begin{definition}\label{definition-exceptional-curve-first-kind}
  Let $X$ be a normal projective surface over $k$.
  An integral curve $C \subset X$ is called an {\it exceptional curve of the first kind} (or simply a {\it $(-1)$-curve}) if $X$ is regular along $C$ and $K_X \cdot_{k_C} C = C\cdot_{k_C} C = -1$.
\end{definition}

\begin{remark}\label{remark-exceptional-curve-first-kind}
  If $C\subset X$ is a $(-1)$-curve, then it follows from Proposition~\ref{proposition-cont} that $C$ is contractible to a regular point.
  Moreover, we have $C \cong \mathbb P^1_{k_C}$ by \cite[Lemma~10.8 (4)]{Kollar-Sing}.
\end{remark}

The following lemma appears in the proof of \cite[Theorem~1]{Umezu-1981}.
\begin{lemma}\label{lemma-crepant-connected}
  Let $\sigma\colon X\to Y$ be a composite of $(-1)$-curve contractions between regular projective surfaces, and let $D\ge0$ be a divisor on $X$ with $K_{X} + D\equiv 0$.
  Then $\Supp(D)$ and $\Supp(\sigma_*D)$ have the same number of connected components.
\end{lemma}
\begin{proof}
  We may assume that $\sigma\colon X\to Y$ is a contraction of a single $(-1)$-curve $E$.
  Since $K_X + D \equiv 0$, we have $K_Y + \sigma_*D \equiv 0$, and therefore $K_X + D\equiv \sigma^*(K_Y + \sigma_*D)$.
  Write $D = \mu E + R$ with $\mu\ge0$ and $E\not\subset\Supp(R)$.
  Substituting $\sigma^*K_Y \equiv K_X - E$ and $\sigma^*\sigma_*R \equiv R + (R\cdot_{k_E} E)E$ into $K_X + D\equiv \sigma^*(K_Y + \sigma_*D)$ yields $\mu = R\cdot_{k_E} E - 1$.
  If $\mu=0$, then $R\cdot_{k_E} E=1$, so $E$ meets $\Supp(R)$ in exactly one point, and contracting $E$ preserves the number of connected components.
  If $\mu>0$, then $E\subset\Supp(D)$, so contracting $E$ again preserves the number of connected components.
\end{proof}

\subsection{Del Pezzo surfaces}\label{subsection-del-pezzo}
\begin{theorem}[MMP for regular surfaces, {\cite[p.~5]{Tanaka-2018-MMP-excelent-surface}} and {\cite[Theorem~1.2]{tanaka_abundance_2020}}]\label{theorem-mmp-surf}
  Let $X$ be a regular projective surface.
  Then we can run a $K_X$-MMP that terminates with either a regular good minimal model or a Mori fiber space $f\colon Z\to B$.
\end{theorem}

\begin{theorem}[{\cite[Theorem~1.5]{Patakfalvi-Waldron-2022} and \cite[Proposition~5.2]{Bernasconi-Tanaka-22}}]\label{theorem-PW22-1.5}
  If $p\ge5$, then a regular (or more generally normal and Gorenstein) del Pezzo surface is geometrically normal.
  If $p\ge11$, then a regular del Pezzo surface $X$ is geometrically regular.
\end{theorem}

\begin{theorem}[{\cite[Theorems~1.1 and 1.8]{Tanaka-24-Bound}}]\label{theorem-bd-del-pezzo}
  If $X$ is a regular del Pezzo surface, then the linear system $\lvert -12 K_X \rvert$ is very ample over $k$.
  Moreover,
  \begin{itemize}
    \item if $p\ge5$ or $X$ is geometrically reduced, then $K_X^2 \le 9$;
    \item if $p = 3$, then $K_X ^2 \le \max \{ 9, 3^{\epsilon(X/k) + 1}\}$, where $\epsilon(X/k)$ is the thickening exponent introduced in \cite{tanaka_invariants_2021};
    \item if $p = 2$, then $K_X ^2 \le \max \{ 9, 2^{\epsilon(X/k) + 3}\}$.
  \end{itemize}
\end{theorem}

\subsection{Mori fiber spaces}\label{subsection-Mori-fiber-spaces}
Throughout this subsection, unless otherwise stated, $\pi\colon X\to C$ is a Mori fiber space from a regular projective surface $X$ over a field $k$ with $H^0(C,\mathcal{O}_C)=k$.
\begin{proposition}[{\cite[Proposition~2.18]{Bernasconi-Tanaka-22}}]\label{proposition-BT22-2.18}%
  If $\dim C = 1$, then for every (not necessarily closed) point $b\in C$ the fiber $X_b$ is an integral conic in $\mathbb P^2_{\kappa(b)}$ with $H^0(X_b,\mathcal{O}_{X_b})=\kappa(b)$.
  Moreover, if $p\ge3$ and $k$ is separably closed, then $\pi$ is a smooth morphism.
\end{proposition}

\begin{lemma}\label{lemma-MFS-geo-regular-base}
  Let $f\colon X\to C$ be a smooth Mori fiber space from a regular surface $X$ to a curve $C$ with $H^0(C,\mathcal O_C) = k$.
  Then $X$ is geometrically regular over $k$ if and only if $C$ is geometrically regular over $k$.
\end{lemma}
\begin{proof}
  Consider the Frobenius base change $f_{k^{1/p}}\colon X_{k^{1/p}}\to C_{k^{1/p}}$:
  \[
    \begin{tikzcd}
      X_{k^{1/p}} \rar{\sigma'} \dar[swap]{f_{k^{1/p}}} & X \dar{f} \\ C_{k^{1/p}} \rar{\sigma} & C \rlap{.}
    \end{tikzcd}
  \]
  Smoothness and surjectivity are preserved by base change, so $f_{k^{1/p}}$ is smooth and faithfully flat.
  Therefore \cite[Theorem~23.7]{MatsumuraCRT} shows that $X_{k^{1/p}}$ is regular if and only if $C_{k^{1/p}}$ is regular.
  The claim follows.
\end{proof}

  When $\dim C = 1$, we have $\rho(X)=2$, and $\overline{\mathrm{NE}}(X)$ has exactly two extremal rays.
Every vertical curve $\Gamma_v$ satisfies $\Gamma_v\equiv\mu\,\pi^*(b_0)$ for some $\mu>0$, so the vertical classes span one of these rays; we let $\mathfrak{f}\in\overline{\mathrm{NE}}(X)$ be its generator normalized by $-K_X\cdot_k \mathfrak{f}=2$.
Equivalently, $\mathfrak{f} = \tfrac{1}{[k(b_0):k]}\,\pi^*(b_0)$ for one (and hence any) closed point $b_0\in C$.
The class $\mathfrak{f}$ is nef with $\mathfrak{f}^2=0$, and $\mathfrak{f}\cdot_k\Gamma=[K(\Gamma):K(C)]$ for every horizontal integral curve $\Gamma\subset X$.

\begin{lemma}\label{lemma-KXC-square}
  Let $f\colon X\to C$ be a smooth Mori fiber space from a regular projective surface onto a curve with $H^0(C,\mathcal O_C)=k$.
  Then $(K_{X/C})^2=0$; consequently $(K_X)^2=-4\deg_kK_C$.
\end{lemma}
\begin{proof}
  The generic fiber $X_\eta$ is a smooth conic, so it becomes isomorphic to $\mathbb P^1$ over some finite separable extension $L/K(C)$.
  Let $h\colon C'\to C$ be the normalization of $C$ in $L$ and form
  \[
    \begin{tikzcd}
      X':=X\times_C C' \rar{\phi}\dar{f'} & X \dar{f} \\
      C' \rar{h} & C .
    \end{tikzcd}
  \]
  The $L$-point of the generic fiber extends uniquely to a section of $f'$.
  A smooth conic fibration with a section is a $\mathbb P^1$-bundle; hence there is a rank-$2$ vector bundle $\mathcal E'$ on $C'$ such that $X'\cong\mathbb P_{C'}(\mathcal E')$.
  By \cite[Exercise~III.8.4(b)]{Hartshorne-AG} we have $\omega_{X'/C'}\cong\mathcal O_{X'}(-2)\otimes f'^*\det\mathcal E'$, and $\xi:=c_1(\mathcal O_{X'}(1))$ satisfies $\xi^2= f'^*c_1(\mathcal E')\cdot\xi$ by the Grothendieck relation (see \cite[\S A.3]{Hartshorne-AG}); hence
  \[
    (K_{X'/C'})^2=(-2\xi+ f'^*c_1(\mathcal E'))^2=4\xi^2-4\,\xi\cdot f'^*c_1(\mathcal E')=0 .
  \]
  Since $\phi$ is the base change of $h$, we have $K_{X'/C'} \sim \phi^*K_{X/C}$, so $\delta\,(K_{X/C})^2=(K_{X'/C'})^2=0$, where $\delta:=\deg\phi=[L:K(C)]$.
  Since $f^*K_C\equiv(\deg_k K_C)\,\mathfrak f$ and $K_X\cdot_k\mathfrak f=-2$, we conclude
  \[
    (K_X)^2=(K_{X/C})^2+2\,K_{X/C}\cdot_k f^*K_C = -4\deg_k K_C .
    \qedhere
  \]
\end{proof}

\smallskip
The following lemma is a key step in the proof of the main theorem.
\begin{lemma}\label{lemma-reducible-fiber}
  Let $f\colon Y\to B$ be a fibration from a regular projective surface to a regular curve. Suppose that $f$ factors as
  \[
    Y\xrightarrow{\ \epsilon\ } Z\xrightarrow{\ g\ } B,
  \]
  where $\epsilon$ is a composition of contractions of $(-1)$-curves and $g$ is a Mori fiber space.
  Let $F_{t_0}$ be a reducible fiber over some closed point $t_0\in B$. 
  Suppose that $F_{t_0}$ contains a unique $(-1)$-curve $T$ and denote $k_T := H^0(T^\nu, \mathcal O_{T^\nu})$.
  Then either $\mathrm{mult}_T(F_{t_0})\ge2$, or $\mathrm{mult}_T(F_{t_0})=1$ and $[k_T:k(t_0)]>1$.
\end{lemma}
\begin{proof}
  By Proposition~\ref{proposition-BT22-2.18}, the fiber $G_{t_0}:=g^*t_0$ is an integral conic over $\kappa:=k(t_0)$ with $H^0(G_{t_0},\mathcal O_{G_{t_0}})=\kappa$.
  Decompose $\epsilon$ as
  \[
    \begin{tikzcd}[column sep=2.2em]
      Y=Y_N \ar[r,"\epsilon_N"]
      &Y_{N-1}\ar[r,"\epsilon_{N-1}"]
      &\cdots\ar[r,"\epsilon_2"]
      &Y_1\ar[r,"\epsilon_1"]
      &Y_0=Z,
    \end{tikzcd}
  \]
  where each $\epsilon_i\colon Y_i\to Y_{i-1}$ is the blow-up of a regular closed point $P_i\in Y_{i-1}$.
  Since blow-ups centered outside the fiber over $t_0$ do not affect $F_{t_0}$, we may assume that every $P_i$ lies over $t_0$.
  Let $E_i\subset Y_i$ be the exceptional curve of $\epsilon_i$.
  We also denote its strict transform on $Y_j$ by $E_i$ for every $j\geq i$.

  Suppose for contradiction that $F_{t_0}$ has a unique $(-1)$-curve $T$ with $\mathrm{mult}_T(F_{t_0})=1$ and $k_T=\kappa$.
  Then $E_N=T$. Moreover, uniqueness forces $P_i\in E_{i-1}\subset Y_{i-1}$ for every $2\leq i\leq N$; otherwise $E_{i-1}$ and $E_i$ would be disjoint $(-1)$-curves on $Y_i$, and subsequent blow-ups would leave at least one $(-1)$-curve above each of them on $Y$.
  Hence
  \[
    \kappa\subseteq k(P_1)\subseteq\cdots\subseteq k(P_N)=k_T=\kappa,
  \]
  so every $P_i$ is $\kappa$-rational.
  Set $m_i=\mathrm{mult}_{E_i}(F_{t_0})$. Then
  \[
    1\leq m_1\leq\cdots\leq m_N=\mathrm{mult}_T(F_{t_0})=1.
  \]
  Hence $P_1$ is a regular point of multiplicity one on $G_{t_0}$, so $G_{t_0}\simeq\mathbb P^1_\kappa$. Since $(G_{t_0})^2_\kappa=0$, the first blow-up produces two $(-1)$-curves of multiplicity one: $E_1$ and the strict transform of $G_{t_0}$.
  If $N=1$, this already contradicts the uniqueness of $T$. Thus $N\geq2$, and $P_2$ must be the intersection of $E_1$ and the strict transform of $G_{t_0}$; otherwise $Y$ would contain at least two $(-1)$-curves.
  But then $E_2$ has multiplicity $2$, contradicting $m_2=1$.
\end{proof}

\section{Geometric regularity of Mori fiber spaces}\label{section-MF}%
\begin{proposition}\label{proposition-MF-geo-normal}
  Let $k$ be a field of characteristic $p\ge5$.
  Let $f\colon X\to C$ be a smooth Mori fiber space from a regular surface $X$ onto a curve $C$ with $H^0(C, \mathcal{O}_C) = k$.
  Suppose that there exists an effective $\mathbb Q$-divisor  $\Delta$ such that
  \begin{itemize}
  \item[\rm(1)] $-(K_X + \Delta)$ is nef, and 
  \item[\rm(2)]
  for every component $\Gamma$ of $\Delta$ that maps birationally to $C$, we have $\coeff_\Gamma(\Delta) \le 1$.
  \end{itemize}
  Then $X$ is geometrically regular.
\end{proposition}

\begin{proof}
  Assume for contradiction that $X$ is geometrically non-regular.
  By Lemma~\ref{lemma-MFS-geo-regular-base}, the curve $C$ is geometrically non-regular, and hence geometrically non-normal.
  Since $p\ge5$, we have $\deg_k K_C > 0$ by Theorem~\ref{theorem-can-bc}.
  By Lemma~\ref{lemma-KXC-square}, we have $(K_{X/C})^2 = 0$ and thus $(K_X)^2=-4\deg_k K_C<0$.
  Since $-(K_X + \Delta)$ is nef, $\Delta^2 \le (-(K_X + \Delta)+\Delta)^2 = (K_X)^2 < 0$.
  Hence there is a component $\Gamma$ of $\Delta$ such that $(\Gamma)^2_k<0$.
  Since all fibers of $f$ are irreducible, this implies that $\Gamma$ dominates $C$.

 We can write that $\Delta = c\Gamma + \Delta'$ where $c=\coeff_\Gamma(\Delta)$. Let $\nu\colon\Gamma^\nu\to\Gamma$ be the normalization, set $k_\Gamma:=H^0(\Gamma^\nu,\mathcal O_{\Gamma^\nu})$ and denote $\psi:=f\circ\nu\colon\Gamma^\nu\to C$. Set $e_\Gamma :=[K(\Gamma):K(C)]$. Then  $e_\Gamma =\Gamma\cdot_k\mathfrak f$, where $\mathfrak f=\tfrac{1}{[k(b_0):k]}f^*(b_0)$ as in Subsection~\ref{subsection-Mori-fiber-spaces}. 
Since $-(K_X+\Delta)$ is nef, $ -(K_X+\Delta)\cdot_k \mathfrak f = 2 -  c e_{\Gamma} - \Delta' \cdot_k \mathfrak f \geq 0$, and it follows that
  \begin{equation}\label{equation0}
    c\,e_{\Gamma} \le 2 - \Delta' \cdot_k \mathfrak f \le2.
  \end{equation}
Considering the intersection with $\Gamma$, we have
    \begin{align}\label{equation1}
(K_X+  c\,\Gamma)\cdot_{k_\Gamma} \Gamma \leq (K_X+  c\,\Gamma + \Delta')\cdot_{k_\Gamma} \Gamma =  (K_X+\Delta)\cdot_{k_\Gamma} \Gamma \leq 0.
  \end{align}
  In turn, applying the Adjunction formula (Proposition~\ref{proposition-adj}) we obtain
  \begin{equation}\label{equation-adj}
0\geq (K_X+  c\,\Gamma)\cdot_{k_\Gamma} \Gamma = (K_X+ \Gamma)\cdot_{k_\Gamma} \Gamma + (c-1)  (\Gamma)^2_{k_\Gamma} \geq \deg_{k_\Gamma}K_{\Gamma^\nu} + (c-1)  (\Gamma)^2_{k_\Gamma}.
  \end{equation}
 
 If the projection $\Gamma \to C$ is separable, then by  Riemann--Hurwitz \cite[Theorem~7.4.16]{Liu-AGAC}, we have $\deg_kK_{\Gamma^\nu} >0$. Remark that $c\leq 1$ always holds: this is obvious when $e_\Gamma \geq 2$ by (\ref{equation0}); and when $e_\Gamma =1$ this is assumption (2). 
 We derive a contradiction by the inequality (\ref{equation-adj}) since  $(\Gamma)^2_{k_\Gamma}<0$.
 
 \smallskip
 Now assume  $\Gamma \to C$ is inseparable, which implies that $p\mid e_\Gamma$.
Since $X\to C$ is a Mori fiber space, we can take the basis $K_{X},\mathfrak f$ of $\mathrm{NS}(X)_{\mathbb Q}$.
Combining the information $(\mathfrak f)_k^2=0$, $K_{X}\cdot_k\mathfrak f =-2$ and $\Gamma\cdot_k\mathfrak f = e_\Gamma$, we can write that
$$\Gamma\equiv -\frac{e_\Gamma}{2}K_{X}+ b \mathfrak f$$
for some $b\in \mathbb{Q}$.
Set $\deg_k K_C = d$. Then $(K_X)_k^2 = -4d$. Notice that  $\deg_{k_{\Gamma}} (f^*K_C|_{\Gamma}) = \frac{d e_\Gamma}{[k_\Gamma:k]} \in \mathbb{Z}$; this implies $[k_\Gamma:k] \mid d e_\Gamma$.
Considering the intersection numbers with $\Gamma$ over $k_{\Gamma}$, we have,
$$ K_X\cdot_{k_\Gamma} \Gamma = \frac{K_X\cdot_{k} \Gamma}{[k_\Gamma:k]}= \frac{2(de_\Gamma-b)}{[k_\Gamma:k]},\quad 
(\Gamma)^2_{k_\Gamma}= \frac{-e_\Gamma(de_\Gamma -2b)}{[k_\Gamma:k]} =-\frac{e_\Gamma}{2}(K_X\cdot_{k_\Gamma} \Gamma) + \frac{be_\Gamma}{[k_\Gamma:k]} <0.$$
Since the above numbers are integers, $2b$ is also divisible by $[k_\Gamma:k]$.
By the inequality (\ref{equation1}) we have
$$(K_X+  c\,\Gamma)\cdot_{k_\Gamma} \Gamma = \frac{2(de_\Gamma-b) -ce_\Gamma(de_\Gamma -2b)}{[k_\Gamma:k]} \geq \frac{2(de_\Gamma-b) -2(de_\Gamma -2b)}{[k_\Gamma:k]} = \frac{2b}{[k_\Gamma:k]} \leq 0,$$
and thus $~K_X\cdot_{k_\Gamma} \Gamma >0$.
In turn, we conclude that
\begin{itemize}
\item[($*$)]
$-(\Gamma)^2_{k_\Gamma} \geq \frac{e_\Gamma}{2}(K_X\cdot_{k_\Gamma} \Gamma) $; and if moreover $K_X\cdot_{k_\Gamma} \Gamma = 1$ and $e_\Gamma$ is odd, say, $e_\Gamma = p \geq 5$ then $b<0$, thus $-(\Gamma)^2_{k_\Gamma} \geq \frac{e_\Gamma}{2} - \frac{be_\Gamma}{[k_\Gamma:k]} \geq e_\Gamma$.
\end{itemize}
Combining $\deg_{k_\Gamma}K_{\Gamma^\nu}  \geq -2$, the inequality (\ref{equation-adj}) and $ce_\Gamma \leq 2$, we have
$$0\geq \deg_{k_\Gamma}K_{\Gamma^\nu} + (c-1)  (\Gamma)_{k_\Gamma}^2 \geq -2 + (c-1) (\Gamma)^2_{k_\Gamma}  \geq -2 + (\frac{2}{e_\Gamma}-1) (\Gamma)^2_{k_\Gamma},$$
thus
$$ (1-\frac{2}{e_\Gamma})( -(\Gamma)^2_{k_\Gamma}) \leq 2.$$
Then applying the statement ($*$), we obtain that
$$(1-\frac{2}{e_\Gamma})( -(\Gamma)^2_{k_\Gamma})  \geq (\frac{e_\Gamma}{2} -1)(K_X\cdot_{k_\Gamma} \Gamma) \leq 2,$$
this requires $e_\Gamma \leq 5$ and $K_X\cdot_{k_\Gamma} \Gamma = 1$, but then $-(\Gamma)^2_{k_\Gamma} \geq e_\Gamma$, which causes a contradiction by
\[
  (1-\frac{2}{e_\Gamma})( -(\Gamma)^2_{k_\Gamma}) \geq e_\Gamma -2 \geq 3.
  \hfill\qedhere
\]
\end{proof}

In Proposition~\ref{proposition-MF-geo-normal}, both assumption~\textup{(2)} and the smoothness assumption on $f$ are necessary (see Examples~\ref{example-pair} and~\ref{example-MF-nonsmooth}, respectively).

\section{Geometric reducedness of $K$-trivial surfaces}\label{section-proof}
This section is devoted to proving the following result.
\begin{theorem}\label{theorem-geo-int-K-trivial}
  Let $\ell$ be a field of characteristic $p\ge 5$ and let $k:=\ell^{1/p}$ be the Frobenius extension of $\ell$.
  Let $S$ be a projective regular surface over $\ell$ such that $H^0(S,\mathcal{O}_S)=\ell$ and $K_S \equiv 0$.
  Assume that $S$ is geometrically non-reduced.

  Then $p = 5$, and $X:=(S\otimes_{\ell} k)_{\rm red}^{\nu}$ has a unique non-canonical singular point $P$,
  which is a rational singularity of multiplicity $3$, and whose minimal resolution has exceptional locus $E_1 \cup E_2$ with $(E_1)^2_{k} = -3$, $(E_2)^2_{k} = -2$.
\end{theorem}

\begin{proof}
  By the properties in Subsection~\ref{section-field-extension}, $S_{\ell^{\rm sep}}$ is regular, and we only need to prove the statements for $S_{\ell^{\rm sep}}$.
  So we may assume that $\ell$ is separably closed.
  By assumption, $S \times_{\ell} k$ is not reduced.
  Set $X = (S \times_{\ell} k)_{\mathrm{red}}^\nu$ and denote by $\pi\colon X\to S$ the induced morphism.
  By Theorem~\ref{theorem-can-bc}, we have
  \begin{equation}\label{equation-cbf-1}
    0 \equiv \pi^*K_S \sim K_X + (p-1) (\mathfrak{M} + \mathfrak{F}),
  \end{equation}
  where $\mathfrak{M}\ne0$ is movable and $\mathfrak{F}$ is effective.
  Let $\sigma\colon \widetilde{X} \to X$ be a minimal resolution. We have
  \[
    \sigma^* K_X = K_{\widetilde{X}} + \sum_i a_i E_i
  , 
  \]
  where $a_i\ge0$ and $E_i$'s are $\sigma$-exceptional curves.
  Then $\kappa(\widetilde{X}) <0$, and therefore a $K_{\widetilde X}$-MMP yields a birational contraction $\epsilon\colon \widetilde X\to Z$, obtained by contracting $(-1)$-curves, onto a Mori fiber space $g\colon Z\to B$:
    \begin{equation}\label{equation-base-change-diagram}
      \begin{tikzcd}
        & & &\widetilde{X}\ar[ld, "\sigma"'] \ar[rd, "\epsilon"] \ar[rdd, "f"'] &\\
      &S \ar[d] & X = (S \times_{\ell} k)_{\rm red}^\nu \ar[l, "\pi"'] \ar[d] & &Z\ar[d, "g"]\\
      &\Spec \ell &\Spec k\ar[l] & &B \rlap{\,.}
    \end{tikzcd}
  \end{equation}  
  Recall that, for a variety $Y$ over $k$, the field $k_Y$ denotes the algebraic closure of $k$ in $K(Y)$.
  In the present situation, we have $k_X = k_{\widetilde{X}} = k_B = k$.

  We shall take advantage of the following two conditions:
  \begin{itemize}
    \item[(a)] Since $S$ is regular and $X \to S$ is of height one, the Cartier index of any Weil divisor on $X$ divides~$p$.
      In particular, for every irreducible divisor $C$ on $X$, if we write $\sigma^*C = \widetilde{C} + \sum_i e_i E_i$, where $\widetilde{C}$ denotes the strict transform of $C$, then $e_i \in \frac1p \mathbb{Z}_{\geq 0}$; and for any two distinct irreducible divisors $C_1$ and $C_2$, if $C_1 \cdot_{k} C_2 >0$, then $C_1 \cdot_{k_{C_1}} C_2 \geq \frac1p$ (which is equivalent to $C_1 \cdot_{k} C_2 \geq \frac1p[k_{C_1}:k]$).
    \item[(b)] We write $\sigma^*(K_X) = K_{\widetilde{X}} + \sum_i a_i E_i$, 
      and $\sigma^* (\mathfrak{M} + \mathfrak{F}) = \widetilde{\mathfrak{M} } + \widetilde{\mathfrak{F}} + \sum_i b_i E_i$, where $\widetilde{\mathfrak{M} } = \sigma^{-1}_*\mathfrak{M}$, $\widetilde{\mathfrak{F} } = \sigma^{-1}_*\mathfrak{F}$, and $a_i, b_i\in \frac1p \mathbb{Z}_{\geq 0}$.
      Then
    \begin{equation}\label{equation-point-b}
      K_{\widetilde{X}} + (p-1) (\widetilde{\mathfrak{M} } + \widetilde{\mathfrak{F}}) + \sum_i m_i E_i \equiv 0
    \end{equation}
      where $m_i := a_i + (p-1)b_i$.
      Since $K_S$ is Cartier, $m_i \in \mathbb{Z}_{\geq 0}$.
\end{itemize}

\begin{lemma}\label{lemma-dim-base-one}
  The Mori fiber space $Z\to B$ is a fibration onto a curve $B$.
\end{lemma}
\begin{proof}
  Suppose to the contrary that $Z$ is a del Pezzo surface with $\rho(Z) = 1$.
  Pushing down (\ref{equation-cbf-1}), we have \[
    K_Z + \epsilon_*\bigl((p-1)\sigma^* (\mathfrak{M} + \mathfrak{F}) + \sum_i a_i E_i\bigr) \equiv 0.
  \]
  Since $\mathfrak{M}$ is movable, the divisor $\epsilon_*(\sigma^* \mathfrak{M}) >0$ is nonzero and effective, and thus ample.
  Consequently,
  $$(-K_Z)_k^2 \geq \bigl( \epsilon_*((p-1)\sigma^* \mathfrak{M}) \bigr)_k^2 \geq (p-1)^2 \geq 16,$$
  which contradicts that $(K_{Z})_k ^2 \le 9$ (see Theorem~\ref{theorem-bd-del-pezzo}).
\end{proof}

\begin{lemma}\label{lemma-E1-horizontal}
    Each component of $\widetilde{\mathfrak{M} } + \widetilde{\mathfrak{F}}$ is contained in a fiber of $f$, and the divisor $\sum_i m_i E_i$ dominates~$B$.
\end{lemma}
\begin{proof}
  Let $\eta\in B$ be the generic point.
  By the condition (b), we have 
  $$K_{\widetilde{X}_{\eta}} + \bigl((p-1) (\widetilde{\mathfrak{M} } + \widetilde{\mathfrak{F}}) + \sum_i m_i E_i\bigr)\big|_{\widetilde{X}_{\eta}} \equiv 0.$$
  Since $\deg_{k(\eta)} K_{\widetilde{X}_{\eta}} = -2$ and $p\geq 5$, we conclude that $ (\widetilde{\mathfrak{M} } + \widetilde{\mathfrak{F}})|_{\widetilde{X}_{\eta}} =0$; that is, every component of $\widetilde{\mathfrak{M} } + \widetilde{\mathfrak{F}}$ is vertical, and it follows that $\sum_i m_i E_i$ dominates $B$.
\end{proof}

For a closed point $b\in B$, denote by $F_b=f^*b$ the fiber of $f$ over $b$, by $k(b)$ the residue field of $b$, and set $\overline{F}_b=\sigma_*F_b$.
A general fiber of $f$ is integral and normal, with $H^0(F_b, \mathcal{O}_{F_b}) = k(b)$.
Since $\mathfrak{M}$ is movable, the strict transform of a general member $M_0\in |\mathfrak{M}|$ contains a fiber $F_b$. Therefore,
\begin{itemize}
  \item[\rm(c)] there exist a reduced and normal fiber $F_b$ and an effective Weil divisor $\overline{V}$ such that $\mathfrak{M} + \mathfrak{F} = \overline{F}_b + \overline{V}$.
\end{itemize}
We fix such a fiber $F_b$ and the divisor $\overline{V}$, and denote by $V$ the strict transform of $\overline{V}$.
Let $E_1$ be a $\sigma$-exceptional curve that dominates $B$; it exists by Lemma~\ref{lemma-E1-horizontal} and we have $m_1>0$.

\begin{lemma}\label{lemma-on-E1}
  Among the $\sigma$-exceptional curves with $m_i>0$, $E_1$ is the unique one dominating $B$.
  Moreover, $F_b\cdot_{k(b)}E_1=1$, $m_1=2$, $F_b\cdot E_i=0$ for every $i\ne1$ with $m_i>0$, and $F_b\cdot V=0$.
  Consequently $k_{E_1}=k$, $E_1$ is a section of $f$ (hence is regular), and $V$ and every $E_i$ with $i\ne1$ and $m_i>0$ are vertical.
\end{lemma}
\begin{proof}
  By restricting the relation $K_{\widetilde{X}} + (p-1) (F_b + V) + \sum_i m_i E_i \equiv 0$ to $F_b$ and applying the adjunction formula, we obtain
  \[
     K_{F_b} + \Bigl((p-2)F_b+(p-1) V + m_1E_1+ \sum_{i\geq 2} m_i E_i\Bigr)\big|_{F_b}\equiv 0.
  \]
  By taking the degree and using the facts that $F_b\cdot_{k(b)} F_b = 0$ and $\deg_{k(b)} K_{F_b}=-2$, we obtain
  \begin{equation}\label{equation-1}
    m_1E_1 \cdot_{k(b)} F_b + (p-1)V\cdot_{k(b)} F_b + \sum_{i\geq 2} m_i E_i\cdot_{k(b)} F_b = 2 .
  \end{equation}
  By restricting the same relation to $E_1$ we obtain
  \[
     K_{E_1^{\nu}} + \Delta_{E_1^{\nu}} + \Bigl((m_1-1)E_1+ (p-1)F_b+(p-1) V + \sum_{i\geq 2} m_i E_i\Bigr)\big|_{E_1}\equiv 0,
  \]
  where $\Delta_{E_1^\nu}\ge0$ is the boundary on $E_1$ obtained from the adjunction formula.
  Taking degree gives
  \begin{equation}\label{equation-2}
        (m_1-1)(E_1)^2_{k_{E_1}} + (p-1)F_b\cdot_{k_{E_1}} E_1 +
        (p-1) V\cdot_{k_{E_1}} E_1 + \sum_{i\geq 2} m_i E_i \cdot_{k_{E_1}} E_1 \le 2.
  \end{equation}
  Since $(p-1)F_b\cdot_{k_{E_1}} E_1 \geq 4$, we conclude from (\ref{equation-2}) that $m_1 \geq 2$.
  Combining this with (\ref{equation-1}), we have
  \begin{align}\label{equation-3}
    m_1 =2, \quad
    E_1 \cdot_{k(b)} F_b=1, \quad\text{and}\quad
    V\cdot F_b = \sum_{i\geq 2} m_i E_i\cdot_{k(b)} F_b=0.
  \end{align}
  Moreover, $E_1$ is the unique $\sigma$-exceptional curve horizontal over $B$ since $K_{\widetilde X} \cdot_{k(b)} F_b = -2$.
  Every $E_i$ (for $i\ge2$ and $m_i>0$) is therefore vertical, hence disjoint from the general fiber $F_b$, which gives $F_b\cdot E_i=0$. 
  Finally, since $E_1 \cdot_{k(b)} F_b=1$, the curve $E_1$ is a section of $f\colon\widetilde{X} \to B$; in particular $k_{E_1}=k$ and $E_1$ is regular.
\end{proof}

\smallskip
Let $P = \sigma(E_1) \in X$ and denote by $E_1, \dots, E_r$ the exceptional curves centered at $P$.
Write $\sigma^*\overline F_b = F_b + \sum_{i} c_i E_i$ and $\sigma^*\overline V = V + \sum_i d_i E_i$, where $c_i, d_i \in \frac1p\mathbb{Z}_{\geq 0}$ by (a).
Then $m_i = a_i + (p-1) (c_i + d_i)$, with $a_i$ as in (b).

\begin{lemma}\label{lemma-ci-positive}%
  \begin{enumerate}[\rm(1)]
    \item $c_i>0$ for all $1\le i\le r$; and
    \item $a_1 = 2/p$ or $(p+1)/p$.
  \end{enumerate}
\end{lemma}
\begin{proof}
  (1) 
  Intersecting the pullback $\sigma^*\overline F_b$ with $E_j$ gives
  \begin{equation}\label{equation-pullback-cut}
    0 = \sigma^*\overline F_b \cdot_k E_j = F_b \cdot_k E_j + \sum_{i=1}^r c_i\,(E_i\cdot_k E_j),
    \qquad 1\le j\le r.
  \end{equation}
  For $j=1$, since $F_b \cdot_k E_1 > 0$ and $(E_1)^2_k<0$, we must have $c_1 > 0$.
  Suppose $T:=\{\,i:c_i=0\,\}\neq\varnothing$. Since $\bigcup_i E_i$ is connected and $1\notin T$, there exist $j_0\in T$ and $i_0\notin T$ with $E_{i_0}\cdot_k E_{j_0}>0$.
  Evaluating (\ref{equation-pullback-cut}) at $j_0$ then forces $c_{i_0}(E_{i_0}\cdot_k E_{j_0})\le0$, a contradiction. Hence $c_i>0$ for all $i$.

  (2) Since $2 = m_1 = a_1 + (p-1)(c_1 + d_1)$, we see that $(c_1 + d_1) = 1/p$ or $2/p$, and the statement follows.
\end{proof}

\begin{lemma}\label{lemma-E1-numerics}
  If $a_1 = 2/p$, then we have $p=5$, $k(b)=k_{E_1}=k$, $p_a(E_1)=0$, $\deg_k K_{E_1}=-2$, $F_b\cdot_k E_1 = 1$, $V\cdot_k E_1 =0$,
  and $(E_1)^2_k \in \{-2,-3\}$.
\end{lemma}
\begin{proof}
By restricting the relation $\sigma^*(K_X) = K_{\widetilde{X}} + a_1E_1 + \sum_{i\geq 2} a_i E_i$ to $E_1$ and applying the adjunction formula, we obtain
$$K_{E_1} + \bigl((a_1-1)E_1 + \sum_{i\geq 2} a_i E_i\bigr)\big|_{E_1} \equiv 0.$$
Taking the degree gives
\begin{equation}\label{equation-4}
  \deg_{k} (K_{E_1}) + (a_1-1)(E_1)^2_{k} + \sum_{i\geq 2} a_i E_i\cdot_{k} E_1 = 0.
\end{equation}
We deduce that 
$$(E_1)^2_k = \frac{p}{p-2}\bigl(\deg_k K_{E_1} + \sum_{i\geq 2} a_i E_i\cdot_k E_1\bigr) \geq \frac{p}{p-2}(\deg_k K_{E_1}) \geq \frac{-2p}{p-2}.$$
Since $(E_1)^2<0$, we have $\deg_k K_{E_1} <0$, hence $p_a(E_1) = 0$ and $\deg_k K_{E_1}=-2$.
Moreover, since $\sigma$ is a minimal resolution, $(E_1)^2_k \le -2$, hence
\begin{equation}\label{equation-E1square}
     (E_1)^2_k = 
     \begin{cases}
        -3 \text{ or } {-2} & \text{if } p = 5,\\
        -2 & \text{if } p \ge 7.
     \end{cases}
\end{equation}
Furthermore, by Lemma~\ref{lemma-on-E1}, the inequality (\ref{equation-2}) can be rewritten as
  \begin{equation}\label{equation-2-prime}
    (E_1)^2_k + (p-1)F_b\cdot_k E_1 +(p-1) V\cdot_k E_1 + \sum_{i\geq 2} m_i E_i\cdot_k E_1 \leq 2.
  \end{equation}
  Combining~(\ref{equation-2-prime}) with (\ref{equation-E1square}) and $F_b\cdot_k E_1 \ge 1$, we conclude that $p=5$, $F_b\cdot_k E_1 = 1$, and $V\cdot_k E_1 =0$.
Finally, $F_b \cdot_k E_1 = 1$ implies $k(b) = k$.
\end{proof}

\begin{lemma}\label{lemma-classification} 
  \begin{enumerate}[\rm(1)]
    \item $\overline{V}\cdot \overline{F}_b =0$;
    \item $\overline{F}_b\cdot_{k(b)} \overline{F}_b = 1/p $ or $2/p$; and
    \item $V\cdot_k E_i = 0$ for $i=1,\ldots,r$ (in particular, $d_i = 0$ for $i=1,\ldots,r$).
  \end{enumerate}
\end{lemma}
\begin{proof}
  As before, write $\sigma^*\overline{F}_b = F_b + c_1E_1 + \sum_{i\geq 2} c_i E_i$ and $\sigma^*\overline{V} = V+ d_1E_1 + \sum_{i\geq 2} d_i E_i$.
  Then
  \begin{equation}\label{equation-4-prime}
  \begin{aligned}
    \overline{F}_b\cdot_{k(b)} \overline{V} &= F_b\cdot_{k(b)} \sigma^*\overline{V}=
      d_1, \quad\text{and}\\
      (\overline{F}_b)^2_{k(b)} &= \sigma^*\overline{F}_b\cdot_{k(b)} F_b = c_1 .
  \end{aligned}
  \end{equation}

  (1) Assume $\overline{V}\cdot \overline{F}_b = d_1 > 0$ for a contradiction. 
  Since $c_1>0$ (Lemma~\ref{lemma-ci-positive}\,(1)) and $2 = m_1 = a_1 + (p-1)(c_1 + d_1)$, we have $c_1=d_1=1/p$ and $a_1 = 2/p$.
  By Lemma~\ref{lemma-E1-numerics}, $p=5$, $F_b\cdot_k E_1=1$ and $(E_1)^2_k\in \{-2,-3\}$.
  Intersecting $\sigma^*\overline F_b$ with $E_1$ now gives
  \[
    0=\sigma^*\overline F_b\cdot_k E_1 = F_b\cdot_k E_1 + c_1(E_1)^2_k + \sum_{i\geq 2}
    c_i\,E_i\cdot_k E_1 \geq 1 - \tfrac{3}{5} > 0,
  \]
  a contradiction. Therefore $\overline{V}\cdot \overline{F}_b = 0$.

  \medskip
    (2) By (\ref{equation-4-prime}) and~(1), we have $(\overline{F}_b)^2_{k(b)}=c_1$ and $d_1=0$, so
$2=m_1=a_1+(p-1)c_1$.
Since $a_1\ge0$, $c_1\in\frac1p\mathbb{Z}_{>0}$, and $3(p-1) >2p$ for $p\ge5$, we obtain $(\overline{F}_b)^2_{k(b)}=c_1=1/p$ or $2/p$.

\medskip
(3) Since $\overline{V}\cdot_k\overline{F}_b = 0$ and $F_b\cdot_k V = 0$, we have
\[
  0 = \overline F_b\cdot_k\overline V = \sigma^*\overline F_b\cdot_k V = F_b\cdot_k V + \sum_i c_i\,(E_i\cdot_k V) = \sum_i c_i\,(E_i\cdot_k V).
\]
Each term is nonnegative and $c_i>0$ by Lemma~\ref{lemma-ci-positive}\,(1), so $V\cdot_k E_i = 0$ for $i=1,\ldots,r$.
Since $\sigma^*\overline{V} = V+ \sum_{i\geq 1} d_i E_i$, for each $j=1,\ldots,r$ we have $0 = \sum_{i=1}^r d_i E_i \cdot_k E_j$, whence $d_i = 0$ by the negative definiteness of the intersection matrix of the $E_i$.
\end{proof}

For convenience, we write
\begin{equation}\label{equation-disc-can}
  K_{\widetilde X}+\mathbf D\equiv0,\qquad
  \mathbf D = (p-1)(F_b+V)+E'+E'',\qquad
  E'=\sum_{i=1}^r m_iE_i,
\end{equation}
where $E'$ is the exceptional divisor centered at $P$, and $E''$ is the sum of the exceptional divisors centered away from $P$.

\begin{lemma}\label{lemma-D-is-connected}%
  \begin{enumerate}[\rm(1)]
    \item $\Supp\mathbf D$ is connected.
    \item $V=0$ and $E''=0$; in particular, $P$ is the only non-canonical singularity on $X$.
  \end{enumerate}
\end{lemma}
\begin{proof}
  (1) The contraction $\epsilon\colon\widetilde X\to Z$ from running the $K_{\widetilde X}$-MMP is a composite of $(-1)$-curve contractions, so by Lemma~\ref{lemma-crepant-connected} it preserves the number of connected components of $\Supp\mathbf D$.
  Since $E_1$ is the only horizontal component of $\mathbf D$, $\epsilon_*\mathbf D$ is the horizontal curve $\epsilon_*E_1$ together with vertical components; since the fibers of $Z\to B$ are irreducible, each such vertical component is a whole fiber meeting $\epsilon_*E_1$. Hence $\Supp(\epsilon_*\mathbf D)$, and therefore $\Supp(\mathbf D)$, is connected.

  (2) By Lemma~\ref{lemma-classification}\,(3), $V$ is disjoint from $E'$; moreover $V$ contains no component of the integral fiber $F_b$ (otherwise $\sigma^*\overline V\supseteq\sigma^*\overline F_b$ would give $d_1\ge c_1>0$), so $V$, being vertical, is disjoint from $F_b$.
  The curves of $E''$ lie in fibers different from $F_b$ and, being exceptional over points different from $P$, are disjoint from $E'$.
  Hence $V\cup E''$ does not meet $F_b\cup E'$; so (1) forces $V = 0$ and $E'' = 0$.
\end{proof}

By Lemma~\ref{lemma-classification}\,(3), we have $d_1=0$, so $(\overline F_b)^2_{k(b)}=c_1$ and $2 = m_1 =a_1+(p-1)c_1$; hence we fall into one of the following two cases:
\begin{itemize}
  \item Case\,(1) $(\overline F_b)^2_{k(b)}=2/p\iff a_1=2/p$;
  \item Case\,(2) $(\overline F_b)^2_{k(b)}=1/p\iff a_1=(p+1)/p$.
\end{itemize}

\begin{proposition}\label{proposition-2/p}
  In Case\,{\rm(1)}, we have $p=5$, and $X$ has a unique non-canonical singular point $P$, which is a rational singularity of multiplicity $3$ whose minimal resolution has exceptional locus $E = E_1 + E_2$ with $(E_1)^2_{k} = -3$ and $(E_2)^2_{k} = -2$.
\end{proposition}
\begin{proof}
  In this case $a_1 = c_1 = 2/p$.
  From Lemma~\ref{lemma-E1-numerics} we have $p=5$, $p_a(E_1)=0$, $F_b\cdot_k E_1 = 1$, $(E_1)^2_k = -2$ or $-3$.
  Intersecting $\sigma^* \overline{F}_b = F_b + \sum_{i=1}^r c_i E_i$ with $E_1$, we obtain \[
    F_b\cdot_k E_1 + \frac{2}{5}(E_1)^2_k + \sum_{i\geq 2} c_i E_i\cdot_k E_1 =0,
  \]
  and thus \[
    -\frac{2}{5}(E_1)^2_k = 1 +\sum_{i\geq 2} c_i E_i\cdot_k E_1 \ge 1.
  \]
  Therefore, $(E_1)^2_k = -3$ and thus $\sum_{i\geq 2} c_i E_i\cdot_k E_1 = 1/p$.
  Since $c_i\in \frac1p \mathbb Z_{\ge0}$, there exists exactly one $E_i$, say $E_2$, such that $E_2 \cdot_k E_1 > 0$ and $c_2 >0$,
  and more precisely $E_2 \cdot_k E_1 = 1$ and $c_2 = 1/p$.
  Here $E_2\cdot_k E_1 =1$ implies $k_{E_2}= k_{E_1} = k$.

Since $m_2 \geq (p-1)c_2 \ge (p-1)/p$ and $m_2\in \mathbb{Z}$, inequality (\ref{equation-2-prime}) then implies $m_2=1$ and $\sum_{i\geq 3} m_i E_i\cdot_k E_1 =0$.
Restricting the same relation to $E_2$ and using $F_b\cdot_k E_2 = V\cdot_k E_2 = 0$ and $E_1\cdot_k E_2 = 1$, we obtain
\begin{equation}\label{equation-E2}
  \deg_k(K_{E_2^\nu}) + \deg_k(\Delta_{E_2^\nu}) + 2 + \Bigl(\sum_{i\geq 3} m_i E_i\Bigr)\cdot_k E_2 = 0,
\end{equation}
where $\Delta_{E_2^\nu}\ge0$ is the boundary obtained from the adjunction formula along $E_2$.
Since $\deg_k(K_{E_2^\nu}) \ge -2$, it follows that $p_a(E_2^\nu) = 0$, $\Delta_{E_2^\nu} = 0$, and $\sum_{i\geq 3} m_i E_i \cdot_k E_2 = 0$; hence $E_2$ is regular.
Using $\sum_{i\geq 3} m_i E_i \cdot_k E_2 = 0$ and $m_i \ge (p-1)c_i >0$, we see that $r=2$ and the exceptional divisor centered at $P$ is $E'=2E_1+E_2$.
Moreover, by the equation
\[
  0 = E_2 \cdot_k \sigma^* \overline{F}_b =E_2 \cdot_k (F_b+ c_1 E_1 + c_2 E_2) = \frac{2}{p} E_1\cdot_k E_2 + \frac{1}{p} (E_2)^2_k,
\]
we deduce $(E_2)^2_k = -2$.
Finally, since $E_1$ and $E_2$ are of arithmetic genus zero with $E_1 \cdot_k E_2 = 1$, $P$ is a rational singularity; it is then straightforward to verify that it has multiplicity $3$.
\end{proof}

We are left to deal with Case~(2).
\begin{lemma}\label{lemma-E2}
  In Case\,{\rm(2)}, there is some $\sigma$-exceptional curve $E_2$ meeting $E_1$.
\end{lemma}
\begin{proof}
  Assume $E_1$ is the only irreducible component of $\sigma^{-1}(P)$.
  Then \[
    0 = \sigma^*\overline F_b\cdot_k E_1 = F_b\cdot_{k} E_1+c_1(E_1)^2_{k}, 
  \]
  where $F_b \cdot_k E_1 = [k(b):k]$.
  Since $c_1=1/p$, we have $(E_1)^2_{k}=-p [k(b):k]$.
  Restricting $K_{\widetilde X} + (p-1) F_b + 2 E_1$ to the regular curve $E_1$ gives
\[
  \deg_{k} K_{E_1}+(E_1)^2_{k} + (p-1)[k(b):k]=0.
\]
  Then $[k(b):k] = \deg_{k} K_{E_1}$.
Since $k$ is assumed to be separably closed, $[k(b):k]$ is a power of $p$; in particular, $[k(b):k]$ is odd.
This contradicts the fact that $\deg_k K_{E_1}$ is even.
\end{proof}

\begin{lemma}\label{lemma-ob}
  Assume $(\overline{F}_b)^2_{k(b)}= 1/p$.
  Then for every closed point $t\in B$, the fiber $F_t$ has at most one non-$\sigma$-exceptional component $T$ with $\overline T\cdot_k\overline F_b>0$, and any such $T$ has $\operatorname{mult}_{F_t}T=1$.
\end{lemma}
\begin{proof}
  From $\overline F_t\equiv\tfrac{[k(t):k]}{[k(b):k]}\overline F_b$ and $(\overline F_b)^2_{k(b)}=1/p$ we get
  \[
    \overline F_b\cdot_k\overline F_t=\tfrac1p[k(t):k].
  \]
  Write $\overline F_t=\sum_j e_j\overline T_j$, where $T_j$ runs over the non-$\sigma$-exceptional components of $F_t$, with $e_j=\operatorname{mult}_{F_t}T_j\ge1$.
  If $\overline T_j\cdot_k\overline F_b>0$, then condition~{\rm(a)} gives
  \[
    \overline T_j\cdot_k\overline F_b\ge\tfrac1p[k_{T_j}:k]\ge\tfrac1p[k(t):k].
  \]
  Since $\tfrac1p[k(t):k]=\overline F_b\cdot_k\overline F_t=\sum_j e_j\,(\overline T_j\cdot_k\overline F_b)$ is a sum of nonnegative terms with $e_j\ge1$, at most one such $T_j$ occurs, and for it $e_j=1$.
\end{proof}

\begin{proposition}\label{proposition-Case2}%
  Case\,{\rm(2)} does not occur.
\end{proposition}
\begin{proof}
  Suppose to the contrary that Case\,(2) holds, that is, $(\overline F_b)^2_{k(b)}=c_1 =1/p$.
  By Lemma~\ref{lemma-D-is-connected} we have 
  \begin{equation}\label{equation-disc-can-prop}
    K_{\widetilde X} + (p-1)F_b + E'\equiv0,
    \qquad E'=\sum_{i=1}^r m_iE_i .
  \end{equation}
  By Lemma~\ref{lemma-E2}, we have $r \ge 2$.
  Let $F_{t_0}$ be the fiber of $f$ containing $E_2$. 
  Then it is reducible and contains at least one $(-1)$-curve.
  Let $T$ be a $(-1)$-curve in $F_{t_0}$; since $\sigma$ is a minimal resolution, $T$ is different from every $E_i$.
  Intersecting~(\ref{equation-disc-can-prop}) with $T$ therefore gives
  \[
    E'\cdot_{k_T} T= - K_{\widetilde X} \cdot_{k_T} T = 1 .
  \]
  Thus $T$ meets exactly one component $E_{i_0}$ of the exceptional divisor centered at $P$, with $E_{i_0}\cdot_{k_T} T=1$.
  Hence \emph{every} $(-1)$-curve $T\subseteq F_{t_0}$ satisfies $\sigma_*T\ne0$ and $\overline T\cdot_k\overline F_b>0$.
  By Lemma~\ref{lemma-ob}, this implies that $F_{t_0}$ has a unique $(-1)$-curve, of multiplicity one. 
  Lemma~\ref{lemma-reducible-fiber} then gives $[k_T:k(t_0)]>1$.
  On the other hand, the proof of Lemma~\ref{lemma-ob} gives
  \[
    \tfrac1p[k(t_0):k]\ge \overline T\cdot_k\overline F_b
    \ge\tfrac1p[k_T:k]\ge\tfrac1p[k(t_0):k].
  \]
  Thus $[k_T:k]=[k(t_0):k]$, and hence $k_T=k(t_0)$, a contradiction.
\end{proof}

  We are in either Case\,(1) or Case\,(2).
  Proposition~\ref{proposition-Case2} excludes Case\,(2), so we are in Case\,(1), and Proposition~\ref{proposition-2/p} concludes the proof of Theorem~\ref{theorem-geo-int-K-trivial}.
\end{proof}

\section{Generalization and Application}\label{section-application}
\begin{theorem}\label{theorem-geo-int-log-pair}
  Let $k$ be a field of characteristic $p\ge 5$.
  Let $S$ be a projective normal surface over $k$ such that $H^0(S,\mathcal{O}_S) = k$.
  Let $\Delta$ be an effective $\mathbb Q$-divisor on $S$ and assume that
  \begin{itemize}
    \item $-(K_S + \Delta)$ is nef, and
    \item the pair $(S,\Delta)$ is log canonical.
  \end{itemize}
  Then the following hold.
  \begin{enumerate}[\rm(1)]
    \item If $p\ge7$, then $S$ is geometrically integral.
    \item If $p=5$ and $S$ is geometrically non-reduced, then $\Delta=0$, $K_S\equiv 0$, the surface $S$ has at most canonical singularities, and $X:=(\widetilde S\otimes_{k} k^{1/p})_{\rm red}^{\nu}$, where $\widetilde S\to S$ denotes the minimal resolution, has a unique non-canonical singular point $P$,
      which is a rational singularity of multiplicity $3$, whose minimal resolution has exceptional locus $E_1 \cup E_2$ with $(E_1 )^2_{k^{1/p}} = -3$ and $(E_2 )^2_{k^{1/p}} = -2$.
  \end{enumerate}
\end{theorem}
\begin{proof}
  We may assume that $k$ is separably closed.
  Assume that $S$ is geometrically non-reduced.
  Let $\sigma\colon\widetilde S\to S$ be a minimal resolution with exceptional divisors $E_i$, and write
  \[
    \sigma^*K_S = K_{\widetilde S} + \sum a_i E_i, \qquad \sigma^*\Delta = \sigma_*^{-1}\Delta + \sum b_i E_i,
  \]
  where $a_i\ge0$ since $\sigma$ is a minimal resolution, and $b_i\ge0$ since $\Delta$ is effective and $\mathbb Q$-Cartier.
  Set $\widetilde{\Delta} = \sigma_*^{-1}\Delta + \sum (a_i + b_i) E_i$, so that $K_{\widetilde S} + \widetilde{\Delta} = \sigma^*(K_S + \Delta)$ is anti-nef.
  Since $(S,\Delta)$ is log canonical, we have $a_i + b_i \le 1$, hence $(\widetilde S,\widetilde{\Delta})$ is again log canonical.
  Note that $\widetilde S$ is geometrically non-reduced and $H^0(\widetilde S,\mathcal O_{\widetilde S}) = k$.
  In the following, we discuss two cases: (i) $K_{\widetilde S} \not\equiv 0$; (ii) $K_{\widetilde S} \equiv 0$.

  (i) Assume $K_{\widetilde S} \not\equiv 0$.
  Set $D := -(K_{\widetilde S} + \widetilde{\Delta})$, which is nef.
  Then $-K_{\widetilde S} = \widetilde{\Delta} + D$ is pseudo-effective, and thus $K_{\widetilde S}$ cannot be pseudo-effective; otherwise $K_{\widetilde S}\equiv0$, contrary to our assumption.
  Then the $K_{\widetilde S}$-MMP (Theorem~\ref{theorem-mmp-surf}) yields a sequence of $(-1)$-curve contractions $\phi\colon \widetilde S\to Z$, followed by a Mori fiber space $f\colon Z \to B$.
  Set $\Delta_Z:=\phi_*\widetilde{\Delta}$ and $D_Z:=\phi_*D$. Then $-(K_Z+\Delta_Z)\equiv D_Z$ is nef, and $\Delta_Z$ satisfies the same coefficient bounds as $\widetilde{\Delta}$.
  If $\dim B=0$, then $Z$ is geometrically normal by Theorem~\ref{theorem-PW22-1.5}.
  If $\dim B=1$, then $f$ is smooth by Proposition~\ref{proposition-BT22-2.18}, so Proposition~\ref{proposition-MF-geo-normal} shows that $Z$ is geometrically regular.
  Thus $Z$ is geometrically normal in either case.
  Since $\widetilde S\to Z$ is birational, the surface $\widetilde S$ is geometrically reduced, a contradiction.
  So Case~(i) does not occur.

  (ii) Assume $K_{\widetilde S} \equiv 0$.
  Theorem~\ref{theorem-geo-int-K-trivial} shows that $p=5$ and $X = (\widetilde S\otimes_{k} k^{1/p})_{\rm red}^{\nu}$ satisfies the stated description.
  It remains to prove the assertions on $(S,\Delta)$ in (2).
  Since $-(K_{\widetilde S} + \widetilde{\Delta}) \equiv -\widetilde{\Delta}$ is nef and $\widetilde{\Delta}$ is effective, we have $\widetilde{\Delta} = 0$.
  In particular $\Delta = 0$ and $a_i = 0$ for all $i$, so $S$ has canonical singularities.
  Moreover $\sigma^*K_S = K_{\widetilde S} \equiv 0$, so $K_S\equiv0$ by the projection formula.
\end{proof}

\begin{corollary}\label{corollary-fib-codim2}
  Let $k$ be a perfect field of characteristic $p\ge7$.
  Let $X$ be a $\mathbb Q$-factorial projective variety.
  Suppose that there exists an effective $\mathbb Q$-divisor $\Delta$ on $X$ such that $(X,\Delta)$ is strongly $F$-regular and $-(K_X+\Delta)$ is nef with Cartier index not divisible by $p$.
   If the Albanese morphism $a_X\colon X\to A_X$ has relative dimension $\leq 2$ over its image, then $a_X$ is a fibration.
\end{corollary}
\begin{proof}
  Denote by $X\xrightarrow{h} W\xrightarrow{g} A_X$ the Stein factorization of $a_X$.
  Since $(X,\Delta)$ is strongly $F$-regular, it is klt \cite[Theorem~3.3]{Hara-Watanabe-2002}.
  If $\dim X = \dim W$, that is, if $X$ has maximal Albanese dimension, then the conclusion follows from \cite[Proposition~3.2]{EP23} or \cite[Proposition~2.9]{CWZ26a}.
  If $\dim X - \dim W = 1$, then the conclusion follows from \cite[Theorem~1.5]{CWZ26a}.
  Finally, assume $\dim X - \dim W =2$.
  The generic fiber $X_\eta$ over $\eta=\Spec K(W)$ is a surface with $H^0(X_\eta,\mathcal O_{X_\eta})=K(W)$, $(X_\eta,\Delta_\eta)$ is klt, and $-(K_{X_\eta}+\Delta_\eta)$ nef.
  Since $p\ge7$, Theorem~\ref{theorem-geo-int-log-pair}(1) shows that $X_\eta$ is geometrically integral, and hence $h$ is separable.
  Therefore $a_X$ is a fibration by \cite[Theorem~1.3]{EP23}.
  Note that the assumption on the Cartier index of $-(K_X+\Delta)$ is needed to apply \cite{EP23}, and the $\mathbb Q$-factoriality of $X$ is needed to apply \cite{CWZ26a}.
\end{proof}

\section{Examples}\label{section-exam}
\subsection{A geometrically non-reduced $K$-trivial surface in characteristic $5$}\label{subsection-exam-K-trivial-5}%
\begin{theorem}\label{theorem-exam-K-trivial-5}
  Let $\ell = \mathbb F_5(s_1,s_2)$ be a field of characteristic $5$ and degree of imperfection $2$.
  Then there exists a projective regular surface $S$ over $\ell$ with $H^0(S,\mathcal O_S)=\ell$ and $K_S\equiv0$ such that $S$ is geometrically non-reduced over $\ell$.
\end{theorem}

\noindent{\it Construction idea:}
Assume that $S$ is such a surface. Set $k:=\ell^{1/5}$ and let $\widetilde X$ be the minimal resolution of $X:=(S \times_\ell k)_{\rm red}^\nu$:
\[
  \begin{tikzcd}
    \widetilde X \rar{\sigma} &
    X \dar\ar[r,"\pi", "/\mathcal F"'] & S\dar \\
                                       &\Spec k \rar & \Spec\ell \rlap{.}
  \end{tikzcd}
\]
Since $[k:\ell]=5^2$ and $S \times_\ell k$ is non-reduced, we have $\deg(\pi) = 5$ and thus $S \cong X/\mathcal F$ for some rank-$1$ foliation $\mathcal F$ on $X$.
According to Proposition~\ref{proposition-2/p}, we can get a candidate for $\widetilde X$ by blowing up $\mathbf F_3=\mathbb P(\mathcal O\oplus\mathcal O(3))\to\mathbb P^1_k$ as follows:
\begin{equation}\label{eq-blow-up}
  \begin{tikzpicture}[baseline=(current bounding box.center),scale=.78,
      every node/.style={font=\small},line cap=round,line join=round]
    \begin{scope}[xshift=0cm]
      \draw (-1.5,.55)--(1.45,.55) node[right] {$E_1\ (-3)$};
      \draw (-.55,-.7)--(-.55,1.35) node[above right] {$E_2\ (-2)$};
      \draw[brown] (-.8,.2)--(.75,-.55)
        node[right,text=black] {$G_2\ (-1)$};
      \draw (.15,.05)--(.85,-1.15)
        node[right] {$G_1\ (-2)$};
      \fill[black] (-.55,.55) circle (1.35pt);
      \node[above left] at (-.55,.55) {$\widetilde P$};
      \node at (0,-1.65) {$\widetilde X$};
    \end{scope}

    \draw[->] (3,.1)--(4.2,.1) node[midway,above] {$\epsilon_2$};

    \begin{scope}[xshift=5.6cm]
      \draw (-1.1,.55)--(1.15,.55) node[right] {$(-3)$};
      \draw (-.55,-.65)--(-.55,1.2) node[above right] {$(-1)$};
      \draw[brown] (-.8,.15)--(.05,-1.05)
        node[right,text=black] {$(-1)$};
      \fill[brown] (-.55,-.2) circle (2.2pt);
    \end{scope}

    \draw[->] (8,.1)--(9.2,.1) node[midway,above] {$\epsilon_1$};

    \begin{scope}[xshift=10.8cm]
      \draw (-1,.55)--(1.15,.55) node[right] {$(-3)$};
      \draw (-.55,-.65)--(-.55,1.2) node[above right] {$F$};
      \fill[brown] (-.55,-.2) circle (2.2pt);
      \node at (0,-1.65) {$\mathbf F_3$};
    \end{scope}
  \end{tikzpicture}
\end{equation}
There exists a normal surface $\widetilde S$ fitting into the following commutative diagram \[
  \begin{tikzcd}
    \widetilde X \rar{\tilde{\pi}} \dar{\sigma} & \widetilde S\dar{\bar\sigma} \\
    X \rar{\pi} & S
  \end{tikzcd}
\]
such that $\bar\sigma$ is birational and $\tilde{\pi}$ is a universal homeomorphism.
Let $\overline E_i = \tilde{\pi}(E_i)$ for $i=1,2$.
One can show that \[
  K_{\widetilde S} = \bar\sigma^* K_S + 2 \overline E_1 + 3 \overline E_2
  \text{ \ and \ }
  \tilde{\pi}^* \overline E_i = E_i\ (i=1,2).
\]
Let $\mathcal D$ be the foliation on $\widetilde X$ corresponding to $\tilde{\pi}$, and write $\mathcal D\cong\mathcal O_{\widetilde X}(-\mathfrak d)$.
Since $K_{\widetilde X} + 4 F + 2 E_1 + E_2 \equiv 0$, and $K_{\widetilde X} + 4 \mathfrak d \equiv \tilde{\pi}^*K_{\widetilde S}$, 
we must have $\mathfrak d \equiv F + E_1 + E_2$.

{\it 
From the above analysis, we only need to find a foliation $\mathcal D$ on $\widetilde X$ such that 
\begin{enumerate}[\rm(1)]
  \item $-\det(\mathcal D) \equiv F+E_1+E_2$, 
  \item and for the induced foliation $\mathcal F$ on $X$, $X/\mathcal F$ is regular. 
\end{enumerate}
\medskip

An AI system found such a foliation $\mathcal D$.
}

\begin{proof}[Proof of Theorem~\ref{theorem-exam-K-trivial-5}] We break the argument into the following steps.
\bigskip

  \emph{Step 1: The toric model $\widetilde X$.}
  \medskip
  
  Let $N:=\mathbb Z^2$ be the lattice with standard basis $e_1,e_2$, and let $\widetilde X$ be the toric surface over $k$ whose fan has primitive ray generators, in cyclic order,
  \[
    (1,0),\quad(0,1),\quad(-1,-2),\quad(-2,-5),\quad(-1,-3),\quad(0,-1),
  \]
  with corresponding toric divisors $F_0,C_\infty,G_1,G_2,E_2,E_1$; see Figure~\ref{figure-exam-5-fan}.
  \begin{figure}[ht]
  \centering
  \begin{tikzpicture}[scale=1, >=stealth]
    \begin{scope}
      \fill[brown!10] (0,0) -- (200:3.1) arc[start angle=200, end angle=360, radius=3.1] -- cycle;
      \node[brown, font=\scriptsize, align=center] at (-45:4.3)
            {$\theta=\langle u_{G_2},u_{F_0}\rangle$\\[1pt]type $\tfrac15(1,2)$\\[1pt]{\scriptsize(merges $V_5{+}V_3{+}V_2$)}};

      \draw[help lines, gray!35] (-3.4,0) -- (3.4,0);
      \draw[help lines, gray!35] (0,-3.5) -- (0,3.1);
      \fill (0,0) circle (1.3pt);

      \node[font=\small\bfseries] at (45:1.35) {$V_1$};
      \node[font=\scriptsize] at (45:1.85) {\scriptsize$(t,w)$};
      \node[font=\small\bfseries] at (130:1.45) {$V_4$};
      \node[font=\scriptsize] at (130:1.95) {\scriptsize$(t',\eta)$};
      \node[font=\small\bfseries] at (185:1.45) {$V_6$};
      \node[font=\scriptsize] at (185:1.95) {\scriptsize$(\xi,\tau)$};
      \node[font=\small\bfseries] at (217.5:1.45) {$V_5$};
      \node[font=\scriptsize] at (217.5:1.95) {\scriptsize$(\tilde w,\rho)$};
      \node[font=\small\bfseries, brown] at (252.5:1.45) {$V_3$};
      \node[font=\scriptsize, brown] at (252.5:1.95) {\scriptsize$(t',\tilde v)$};
      \node[font=\small\bfseries] at (315:1.45) {$V_2$};
      \node[font=\scriptsize] at (315:1.95) {\scriptsize$(t,v)$};

      \draw[thick,->] (0,0) -- (0:3.0)
        node[right, align=center, font=\small]
        {$F_0$\\[-1pt]{\scriptsize$(1,0)$}, $0$};
      \draw[thick,->] (0,0) -- (0,2.8)
        node[right, align=center, font=\small]
        {$C_\infty$\\[-1pt]{\scriptsize$(0,1)$}, $2$};
      \draw[thick,->] (0,0) -- (170:2.9)
        node[above left, align=center, font=\small]
        {$G_1$\\[-1pt]{\scriptsize$(-1,-2)$}, $-2$};
      \draw[thick,->] (0,0) -- (200:3.1)
        node[left, align=center, font=\small]
        {$G_2$\\[-1pt]{\scriptsize$(-2,-5)$}, $-1$};
      \draw[thick,brown,->] (0,0) -- (235:3.0)
        node[below left, align=center, font=\small]
        {\color{brown}$E_2$\\[-1pt]{\scriptsize$(-1,-3)$}, $-2$};
      \draw[thick,brown,->] (0,0) -- (270:3.2)
        node[below, align=center, font=\small]
        {\color{brown}$E_1$\\[-1pt]{\scriptsize$(0,-1)$}, $-3$};

      \fill[brown] (252.5:0.55) circle (1.6pt);
      \node[font=\scriptsize, brown] at (252.5:0.95) {$\widetilde P$};

      \node[font=\scriptsize, brown, align=left, anchor=west]
            at (-4.5,2.25)
            {highlighted rays deleted by $\sigma$;\\[-1pt]
             $\widetilde P\in V_3$};
    \end{scope}
  \end{tikzpicture}
  \caption{The toric surface $\widetilde X$}
  \label{figure-exam-5-fan}
  \end{figure}
  Let $M:=\operatorname{Hom}(N,\mathbb Z)$ be the dual lattice with dual basis $e_1^\vee,e_2^\vee$, and let
  \[
    t:=\chi^{e_1^\vee},\qquad w:=\chi^{e_2^\vee}
  \]
  be the corresponding characters of the dense torus.
  Each consecutive pair has determinant $1$, so $\widetilde X$ is smooth.
  Consecutive toric divisors meet once and nonconsecutive ones are disjoint; hence \[
    E_1\cdot_kE_2=E_2\cdot_kG_2=G_2\cdot_kG_1=1,
  \]
  while the remaining pairs among $E_1,E_2,G_1,G_2$ are disjoint.
  For a torus-invariant prime divisor $D$, let $u_{\mathrm{prev}}$ and $u_{\mathrm{next}}$ be the primitive generators adjacent to $u_D$.
  The toric relation \[
    u_{\mathrm{prev}}+u_{\mathrm{next}}=-(D)^2_k u_D
  \]
  gives \[
  (E_1)^2_k=-3,\qquad (E_2)^2_k=-2,\qquad
  (G_1)^2_k=-2,\qquad (G_2)^2_k=-1.
  \]
  Set $\widetilde P:=E_1\cap E_2$.

  The rational function $t$ extends to a toric morphism $f\colon\widetilde X\to B\cong\mathbb P^1_k$.
  Choose $b\in B(k)$ in the dense torus and put $F_b:=f^*(b)$.
  Then $E_1$ is a section, $E_2$ is vertical, $F_b\cdot_kE_1=1$, and
  \begin{equation}\label{equation-exam-5-fiber}
    F:=F_b\equiv F_0\equiv E_2+G_1+2G_2,
  \end{equation}
  where the last equivalence follows from $\operatorname{div}(t)=F_0-E_2-G_1-2G_2$.
  Since the toric boundary is anticanonical, we have \[
    -K_{\widetilde X}=F_0+C_\infty+G_1+G_2+E_2+E_1.
  \]
  The relation
  $\operatorname{div}(w)=C_\infty-E_1-2G_1-5G_2-3E_2$ and
  \eqref{equation-exam-5-fiber} then give
  \begin{equation}\label{equation-exam-5-can}
    -K_{\widetilde X}\equiv
    F_0+2E_1+3(E_2+G_1+2G_2)+E_2
    \equiv 4F+2E_1+E_2.
  \end{equation}
  \medskip
  
  \emph{Step 2: The foliation $\mathcal D$.}
  
    \medskip
  Let $c_i = s_i^{1/p}$ ($i=1,2$) so that $k = \ell^{1/p} = \ell(c_1,c_2)$.
  Write $\partial_{c_1},\partial_{c_2}$ for the basis of
  $\operatorname{Der}_\ell(k)$ dual to the $p$-basis $c_1,c_2$; their common
  kernel is $\ell$.
  Set $K:=K(\widetilde X) = k(t,w)$ and let \[
    \delta=\partial_{c_1}+t\,\partial_{c_2}+w\,\partial_t \in \operatorname{Der}_\ell(K).
  \]
  Equivalently, $\delta$ is given on the $p$-basis $c_1,c_2,t,w$ of $K$ over $K^5$ by
  \[
    \delta(c_1)=1,\qquad \delta(c_2)=t,\qquad \delta(t)=w,\qquad \delta(w)=0.
  \]
  Then $\delta^{[5]}=0$, as one checks on the $p$-basis, so $K\delta$ is $p$-closed and $L:=\ker\delta$ has index $5$ in $K$.
  For $\lambda\in k$ one has
  $\delta(\lambda)=\partial_{c_1}(\lambda)+t\,\partial_{c_2}(\lambda)$.
  Using the linear independence of $1,t$ over $k$, together with
  $\delta(c_1)=1$ and $[K:L]=5$, we obtain
  \begin{equation}\label{equation-exam-5-Lk}
    L\cap k=\ell,\qquad Lk=K .
  \end{equation}
  Let $\mathcal D:=T_{\widetilde X/\ell}\cap K\delta$ be the saturation of $\delta$.
  Then $\mathcal D$ is a rank-$1$ foliation on a regular surface $\widetilde X$ and thus locally free, so
  \[
    \mathcal D \cong\mathcal O_{\widetilde X}(-\mathfrak d)
  \]
  for some divisor $\mathfrak d$.
  Put $t'=1/t$, $v=1/w$, $\tilde w=w/t^3$, $\tilde v=1/\tilde w$, $\eta=w/t^2$, $\xi=t^2/w$ and $\rho=t^5/w^2=1/\tau$.
  Clearing denominators and cancelling common factors in the six charts below gives the local generator $\partial$ of $\mathcal D$:
  \begin{center}
  \renewcommand{\arraystretch}{1.2}
  \begin{tabular}{c|c|c|cccc}
    chart & $(y_1,y_2)$ & $\partial$ & $\partial(c_1)$ & $\partial(c_2)$ & $\partial(y_1)$ & $\partial(y_2)$ \\ \hline
    $V_1$ & $(t,w)$ & $\delta$ & $1$ & $t$ & $w$ & $0$ \\
    $V_2$ & $(t,v)$ & $v\delta$ & $v$ & $tv$ & $1$ & $0$ \\
    $V_3$ & $(t',\tilde v)$ & $t'^2\tilde v\,\delta$ & $t'^2\tilde v$ & $t'\tilde v$ & $-t'$ & $3\tilde v$ \\
    $V_4$ & $(t',\eta)$ & $t'\delta$ & $t'$ & $1$ & $-t'\eta$ & $3\eta^2$ \\
    $V_5$ & $(\tilde w,\rho)$ & $\tilde w^2\rho^2\delta$ & $\tilde w^2\rho^2$ & $\rho$ & $2$ & $0$ \\
    $V_6$ & $(\xi,\tau)$ & $\xi^2\tau\,\delta$ & $\xi^2\tau$ & $1$ & $2$ & $0$
  \end{tabular}
  \end{center}
  For example, in the chart $V_5$, we have $t = 1/(\tilde{w}^2\rho)$ and $w = 1/(\tilde{w}^5\rho^3)$, so \[
    \delta=\partial_{c_1}+\frac1{\tilde w^2\rho}\,\partial_{c_2}
    -\frac3{\tilde w^2\rho^2}\,\partial_{\tilde w} +\frac5{\tilde w^3\rho}\,\partial_\rho.
  \]
  Clearing denominators gives the regular (but not saturated) section \(\tilde w^3\rho^2\delta\), with coefficients
  \[
  \bigl(\tilde w^3\rho^2,\ \tilde w\rho,\ -3\tilde w,\ 5\rho\bigr).
  \]
  Then, since $p=5$, we have $5\rho=0$, and we can cancel the common factor $\tilde{w}$ to obtain $\tilde w^2\rho^2\delta$.
  One can work out the other charts similarly.

  Here $E_1=\{v=0\}=\{\tilde v=0\}$; $E_2=\{t'=0\}$ in $V_3$ and $\{\rho=0\}$ in $V_5$; $G_1=\{t'=0\}$ in $V_4$ and $\{\tau=0\}$ in $V_6$; $G_2=\{\tilde w=0\}$ in $V_5$ and $\{\xi=0\}$ in $V_6$; and $F_0,C_\infty\subseteq V_1$.

  Since $\delta(c_1)=1$, evaluation at $c_1$ defines a nonzero section of
  $\mathcal D^\vee\cong\mathcal O_{\widetilde X}(\mathfrak d)$.  Relative to
  each local generator $\partial$ in the table, this section is represented
  by $\partial(c_1)$.  Its vanishing orders along
  $E_1,E_2,G_1,G_2$ are respectively $1,2,1,2$, and it has no other zeros.
  Hence
  \begin{equation}\label{equation-exam-5-d}
    \mathfrak d=E_1+2E_2+G_1+2G_2.
  \end{equation}
  Thus by \eqref{equation-exam-5-fiber}, \[
    \mathfrak d \equiv E_1+E_2+F,\qquad
    \mathfrak d\cdot_kE_1=\mathfrak d\cdot_kE_2=-1,\quad \mathfrak d\cdot_kF=1 .
  \]
  Moreover, $\partial(v)=0$ and $\partial(\tilde v)=3\tilde v$ show that $E_1$ is $\mathcal D$-invariant, while $\partial(t')=-t'$ and $\partial(\rho)=0$ show that $E_2$ is $\mathcal D$-invariant;
   and in each chart except $V_3$ some value of $\partial$ is a unit, while in $V_3$ the four values generate the unit ideal away from $\widetilde P$.
  So $\widetilde P$ is the only point at which $\mathcal D\subseteq T_{\widetilde X/\ell}$ is not a subbundle.

  \medskip
  \emph{Step 3: Contract and descend.}
    \medskip
    
  Deleting the rays $u_{E_1},u_{E_2}$ defines a projective toric contraction $\sigma\colon\widetilde X\to X$; set $P:=\sigma(E_1\cup E_2)$.
  The germ $(X,P)$ is the affine toric surface of the cone $\langle(-2,-5),(1,0)\rangle$, i.e., the cyclic quotient singularity of type $\frac15(1,2)$.
  Thus $X$ is normal projective, regular away from $P$, and $\sigma$ is its minimal resolution.
  Since $\widetilde X\to X$ is birational, $\mathcal D$ induces a foliation $\mathcal F$ on $X$; set
  \[
    S:=X/\mathcal F,\qquad \pi\colon X\to S,\qquad s_0:=\pi(P).
  \]
  Then $S$ is a normal projective surface.

  \medskip
  \emph{Step 4: Regularity of $S$.}
  
    \medskip
  First consider $x\in\widetilde X\setminus(E_1\cup E_2)$ and let $\partial$ generate $\mathcal D$ at $x$.
  By the chart computation in Step~2 there is $g$ with $\partial(g)\in\mathcal O_{\widetilde X,x}^\times$; replacing $\partial$ by $\partial(g)^{-1}\partial$, which spans the same line, we may assume $\partial(g)=1$, and then $\partial^{[5]}=c\partial$ with $c=\partial^5(g)=\partial^4(1)=0$.
  Hence $A:=\ker\partial$ satisfies $\mathcal O_{\widetilde X,x}=\bigoplus_{i=0}^{4}Ag^i$, so $A\to\mathcal O_{\widetilde X,x}$ is a faithfully flat local homomorphism and $A$ is regular by \cite[Theorem~23.7]{MatsumuraCRT}.
  Thus $S\setminus\{s_0\}$ is regular.

  We next show that $S$ is regular at $s_0$.
  On $V_2\subset\widetilde X$, $(t,v)$ are coordinates and $E_1=\{v=0\}$.
  Set $U:=v=1/w$ and $T:=t$.  Via the common dense torus of $\widetilde X$ and $X$, these are rational functions on both surfaces, but not local parameters at $P\in X$.
  On $\widetilde X$ one has $\nu_{E_1}(U^aT^b)=a$ and $\nu_{E_2}(U^aT^b)=3a-b$.
  The monomial $U^aT^b=t^bw^{-a}$ corresponds to the lattice point $(b,-a)\in M$, hence is regular on the affine chart of $X$ for $\langle(-2,-5),(1,0)\rangle$ if and only if $b\ge0$ and $5a-2b\ge0$.
  Set $\mathcal C:=\{(a,b)\in\mathbb Z^2:0\le2b\le5a\}$
  and $A:=k[U^aT^b:(a,b)\in\mathcal C]$, and let
  $\mathfrak m\subset A$ be the maximal ideal corresponding to $P$.
  Set
  \[
    R:=\mathcal O_{X,P}=A_{\mathfrak m}.
  \]
  The generator of $\mathcal D$ at $\widetilde P$ given in Step~2
  can be written in terms of $U,T$ as
  \[
    \partial=t'^2\tilde v\,\delta
      =T\partial_T+U\bigl(T\partial_{c_1}+T^2\partial_{c_2}\bigr).
  \]
  Thus
  \begin{equation}\label{equation-exam-5-partial}
    \partial(\lambda U^aT^b)=\lambda b\,U^aT^b+(\partial_{c_1}\lambda)U^{a+1}T^{b+1}+(\partial_{c_2}\lambda)U^{a+1}T^{b+2}
    \qquad(\lambda\in k),
  \end{equation}
  so $\partial(R)\subseteq R$.
  Since $\partial(U)=0$ and $\partial(T)=T$, we have $\partial^n(c_1)=UT$ and $\partial^n(c_2)=2^{n-1}UT^2$ for $n\ge1$.
  Thus $2^4=1$ implies that $\partial^{[5]}$ and $\partial$ agree on the generators $U,T,c_1,c_2$ of $K/\mathbb F_5$; as both are derivations in characteristic $5$, they are equal.
  Since \(X^5-X\) is separable, \(\partial^{[5]}=\partial\) yields an eigenspace decomposition \(R=\bigoplus_{i\in\mathbb F_5}R_i\) with \(R_i=\{x\in R:\partial(x)=ix\}\); in particular \(R_0=\ker(\partial|_R)=R\cap L=\mathcal O_{S,s_0}\), and the projection \(R\to R_0\) is \(1-\partial^4\).
  Filter $R$ by $F_n=\{x:\nu_{E_1}(x)\ge n\}$, so that $\operatorname{gr}R=\bigoplus_n F_n/F_{n+1}$.
  The formula~\eqref{equation-exam-5-partial} shows $\partial(F_n)\subseteq F_n$ and $(\partial-T\partial_T)(F_n)\subseteq F_{n+1}$.
  Write $\operatorname{gr}\partial$ for the induced derivation of $\operatorname{gr}R$.
  Then $\operatorname{gr}\partial=T\partial_T$, with $\operatorname{gr}\partial\,(U^aT^b)=b\,U^aT^b$ in $\operatorname{gr} R$.
  Hence every element of $\operatorname{gr}R_0$ is a $k$-linear combination of the initial forms of monomials $U^aT^b$ with $5\mid b$.
  Conversely, if $(a,b)\in\mathcal C$, $5\mid b$, and $\lambda\in k$, then the projection $1-\partial^4:R\to R_0$ satisfies
  \[
    (1-\partial^4)(\lambda U^aT^b)=\lambda U^aT^b+(\text{higher }\nu_{E_1}).
  \]
  Thus every such monomial occurs as an initial form in $\operatorname{gr}R_0$.
  In particular the residue field of $R_0$ is $k$.
  The semigroup $\mathcal C\cap\{5\mid b\}$ is freely generated by $(1,0)$ and $(2,5)$, since $(a,5m)=(a-2m)(1,0)+m(2,5)$.
  Set $u:=U$, $r:=U^2T^5$, and let $\mathfrak n$ be the maximal ideal of $R_0$.
  Both $u$ and $r$ lie in $R_0$.  For $0\ne x\in R_0$ with $d=\nu_{E_1}(x)$, write $\operatorname{in}(x)$ for the image of $x$ in $F_d/F_{d+1}$.
  The preceding computation gives
  \[
    \operatorname{gr}R_0=k[\operatorname{in}(u),\operatorname{in}(r)].
  \]
  For the filtration induced by $F_\bullet$ on $\mathfrak n/\mathfrak n^2$, this gives
  \[
    \operatorname{gr}(\mathfrak n/\mathfrak n^2)
    =k\operatorname{in}(u)\oplus k\operatorname{in}(r),
  \]
  where the two generators have degrees $1$ and $2$; all other nonconstant monomials lie in $\mathfrak n^2$.
  Thus $u,r$ generate $\mathfrak n/\mathfrak n^2$, so Nakayama's lemma gives $\mathfrak n=(u,r)$.
  Since $\dim R_0=2$, we obtain
  \begin{equation}\label{equation-exam-5-germ}
    R_0=\mathcal O_{S,s_0}\ \text{is regular, with regular parameters}\quad
    u=\frac1w,\qquad r=\frac{t^5}{w^2} .
  \end{equation}
  Finally, $U^i(U^2T^5)^j=U^{i+2j}T^{5j}$ has $\nu_{E_1}=i+2j$ and $\nu_{E_2}=3i+j$, and distinct $(i,j)$ with equal $i+2j$ have distinct $3i+j$. 
  So no leading terms cancel and $\nu_{E_1},\nu_{E_2}$ are the monomial valuations of weights $(1,2)$ and $(3,1)$ in $u,r$, of discrepancies $1+2-1=2$ and $3+1-1=3$.

  \medskip
  \emph{Step 5: The conclusion.}
  
  \medskip
  Let $\tilde\pi\colon\widetilde X\to\widetilde S:=\widetilde X/\mathcal D$ and $\overline E_i:=\tilde\pi(E_i)$.
  Since $E_i$ is $\mathcal D$-invariant, we have $\tilde\pi^*\overline E_i=E_i$ by \cite[Proposition~1]{Rudakov-Shafarevich-1976}.
  By \eqref{eq:foliation-2}, \eqref{equation-exam-5-can} and \eqref{equation-exam-5-d},
  \[
    \tilde\pi^*K_{\widetilde S}\equiv K_{\widetilde X}+4\mathfrak d
    \equiv-(4F+2E_1+E_2)+4(F+E_1+E_2)=2E_1+3E_2
    =\tilde\pi^*(2\overline E_1+3\overline E_2),
  \]
  so $K_{\widetilde S}\equiv2\overline E_1+3\overline E_2$ since $\tilde{\pi}^*$ is injective on numerical classes.
  Since the discrepancies of $\overline E_1,\overline E_2$ over the regular point $s_0$ are $2$ and $3$ (Step~4), the contraction $\widetilde S\to S$ of $\overline E_1+\overline E_2$ satisfies $K_{\widetilde S}=\bar\sigma^*K_S+2\overline E_1+3\overline E_2$, whence $K_S\equiv0$.

  \smallskip
  Since $\mathcal O_S\subseteq\mathcal O_X$, \eqref{equation-exam-5-Lk} gives $H^0(S,\mathcal O_S)=k\cap L=\ell$.
  The same relation yields $[K:L]=[Lk:L]=5<[k:\ell]$.
  By Mac~Lane's criterion, a finitely generated extension $F/\ell$ is separable if and only if $[F\cdot\ell^{1/p}:F]=[\ell^{1/p}:\ell]$. 
  Thus $L/\ell$ is inseparable, and therefore $S$ is geometrically non-reduced over $\ell$.
\end{proof}

\subsection{Geometrically non-normal $K$-trivial surfaces in characteristic $\le7$}\label{subsection-exam-non-normal-le7}%
\begin{example}\label{example-57}
  Let $k_0$ be an algebraically closed field of characteristic $p\le7$.
  Let $G=\mu_p$ be the infinitesimal group scheme over $k_0$.
  Let $C$ be a smooth curve over $k_0$ with a free $G$-action (e.g., $C$ is an ordinary elliptic curve or $C = \mathbb A^1\setminus \{0\}$).
  By \cite[Theorem~7.3 (2), Examples~10.2, 10.3, 10.7 and 10.8]{Matsumoto-2023-mu_p-alpha_p}, there exists a $G$-quotient $\rho\colon T\to S$ such that 
  \begin{itemize}
    \item $S$ is an RDP K3 surface, i.e., $S$ has at most rational double point singularities and its minimal resolution is a K3 surface; 
    \item $T$ is a proper integral non-normal Gorenstein surface satisfying $\omega_T \cong \mathcal{O}_T$, $\operatorname{Sing}(T) \cap \rho^{-1}(\operatorname{Sing}(S)) = \varnothing$, and
      $\operatorname{Supp}\operatorname{Fix}(G) = \rho^{-1}(\operatorname{Sing}(S))$.
  \end{itemize}
  Let $G$ act diagonally on $T \times C$.
  Since $G$ acts freely on $C$, it acts freely on $T\times C$ and thus the quotient morphism \[
    q\colon T \times C \longrightarrow W := (T \times C)/G
  \]
  is a finite faithfully flat $G$-torsor.
  The projections induce morphisms
  \[
    \begin{tikzcd}
      W = (T \times C)/G \rar{f}\dar{g} & C/G = C' \\
      T/G = S\,. &
    \end{tikzcd}
  \]
  Set $U := S^{\mathrm{reg}}$ and $T_U := \rho^{-1}(U)$.
  Since $T_U\to U$ is a $G$-torsor, the associated bundle
  \[
    g^{-1}(U) = (T_U \times C)/G \longrightarrow U
  \]
  is fppf locally the smooth projection $T_U \times C \to T_U$; hence it is smooth and $g^{-1}(U)$ is regular.
  For $w\in g^{-1}(S\setminus U)$, choose $(t,c)\in T\times C$ above $w$.
  Since $\rho(t)\in\operatorname{Sing}(S)$, the stated property of $T$ shows that $T$ is regular at $t$, so $T\times C$ is regular at $(t,c)$.
  Regularity descends along $q$ by \cite[\href{https://stacks.math.columbia.edu/tag/07NG}{Tag~07NG}]{stacks-project}, so $W$ is regular at $w$ and hence everywhere.

  Set $K:=K(C')$ and $L:=K(C)$, and let $X:=W_K$ be the generic fiber of $f$.
  The morphism $f$ is flat and projective. Since $W$ is regular, so is $X$.
  By construction, $X_L\cong T_L$, so $X$ is geometrically integral and geometrically non-normal.
  Moreover, $H^0(X,\mathcal O_X)=K$.
  Finally, since \[
    \omega_X|_{X_L}\cong\omega_{X_L}\cong(\omega_T)_L\cong\mathcal O_{X_L}
  \]
  and $\operatorname{Pic}(X)\to\operatorname{Pic}(X_L)$ is injective by \cite[\href{https://stacks.math.columbia.edu/tag/0CC5}{Tag~0CC5}]{stacks-project}, we have $\omega_X\cong\mathcal O_X$.
  Thus $X$ is a projective, regular, geometrically integral, geometrically non-normal, $K$-trivial surface.
\end{example}

\subsection{Counterexamples for Mori fiber spaces}\label{subsection-exam-MF-hypotheses}%
The following example shows that the coefficient bound in Proposition~\ref{proposition-MF-geo-normal} is necessary.
For each prime $p$, we construct a geometrically non-reduced Mori fiber space $f\colon X\to C$ with an effective $\mathbb Q$-divisor $\Delta$ such that $(X,\Delta)$ is not log canonical.
\begin{example}\label{example-pair}
  Let $k=\mathbb F_p(s,t)$ and let \[
    C=(X_0^p+sX_1^p+tX_2^p=0)\subseteq\mathbb P^2_k.
  \]
  Then $C$ is a regular, geometrically non-reduced projective curve with $H^0(C,\mathcal O_C)=k$.
  For $n>0$, set $\mathcal E=\mathcal O_C\oplus\mathcal O_C(-n)$ and $X=\mathbb P_C(\mathcal E)$.
  Then $\deg_k\mathcal E=-np$, and the projection $f\colon X\to C$ is a smooth Mori fiber space;
  moreover, $X$ is geometrically non-reduced.
  Let $C_0$ be the section corresponding to $\mathcal E\twoheadrightarrow\mathcal O_C(-n)$ \cite[Proposition~II.7.12]{Hartshorne-AG}.
  Then $\mathcal O_X(C_0)\cong\mathcal O_X(1)$, and the relative Euler sequence gives \[
    \omega_X\cong \mathcal O_X(-2C_0)\otimes f^*(\omega_C\otimes\det\mathcal E).
  \]
  Hence \[
    K_X+2C_0+(np-\deg_k\omega_C)\mathfrak f\equiv0, \qquad
    \mathfrak f:=\frac{f^*(b)}{[k(b):k]}
  \]
  for any closed point $b\in C$.
  Choose $n$ with $np\ge\deg_k\omega_C$ and set \[
    \Delta:=2C_0+(np-\deg_k\omega_C)\mathfrak f.
  \]
  Then $\Delta\ge0$ and $-(K_X+\Delta)\equiv0$, while $\coeff_{C_0}(\Delta)=2$.
  Thus $(X,\Delta)$ is not log canonical.
\end{example}

The following example shows that the smoothness assumption in Proposition~\ref{proposition-MF-geo-normal} is also necessary.
\begin{example}\label{example-MF-nonsmooth}
  Let $p > 2$, let $a,d$ be algebraically independent over $\mathbb F_p$, and set $k:=\mathbb F_p(a,d)$ and $C:=\mathbb P^1_k$ with homogeneous coordinates $[U:V]$.
  Let $\mathcal E:=\mathcal O_C(-p)^{\oplus2}\oplus\mathcal O_C$, write $\pi\colon \mathbf P:=\mathbb P_C(\mathcal E)\to C$, and let $x,y,z$ be the tautological coordinates associated with the three summands, so that $x,y$ are sections of $\mathcal O_{\mathbf P}(1)\otimes\pi^*\mathcal O_C(p)$ and $z$ is a section of $\mathcal O_{\mathbf P}(1)$.
  Let $G:=(U^p-aV^p)(U^p-UV^{p-1})\in H^0(C,\mathcal O_C(2p))$.
  Then $x^2-dy^2-Gz^2$ is a section of $\mathcal O_{\mathbf P}(2)\otimes\pi^*\mathcal O_C(2p)$, and we define \[
    f\colon X:=(x^2-dy^2-Gz^2=0)\subset \mathbf P\longrightarrow C.
  \]
  The zero divisor $Z(G)\subset C$ is reduced.
  For each $b\in Z(G)$, the fiber $X_b$ has the unique singular point
  $P_b=[0:0:1]$. On the chart $z\ne0$, the equation near $P_b$ is
  $x^2-dy^2=tu$, where $t$ is a uniformizer at $b$ and $u$ is a unit.
  Hence $X$ is regular at $P_b$, and it is smooth elsewhere.
  Since every fiber is integral, $\rho(X/C)=1$, and $-K_X$ is relatively ample. Thus $f$ is a Mori fiber space,
  but it is not smooth.

  Let $b_a\in C$ be the closed point defined by $U^p-aV^p=0$ and
  $P_{b_a}=[0:0:1]\in X_{b_a}$. After adjoining $\alpha:=a^{1/p}$ and
  $\delta:=\sqrt d$, work on the charts $V\ne0$ and $z\ne0$ and set
  \[
    w:=U/V,\qquad s:=w-\alpha,\qquad r:=x-\delta y,\qquad h:=w^p-w.
  \]
  Since $h$ is a unit at $w=\alpha$, putting $q:=(x+\delta y)/h$ gives the
  following completed local equation at the point above $P_{b_a}$:
  \[
    rq=s^p,
  \]
  which is a rational double point of type $A_{p-1}$. Hence $X$ is geometrically
  normal but not geometrically regular.

  Set $\xi:=\mathcal O_{\mathbf P}(1)|_X$, let $F$ be the class of a fiber of $f$ over some $k$-closed point of $C$, and put $D:=\{z=0\}$. Then $D\sim\xi$ is an integral
  bisection, and adjunction gives $K_X=-\xi-2F$.
  For $\Delta:=\bigl(1-\frac2p\bigr)D$ we have
  \[
    -(K_X+\Delta)=\frac2p(\xi+pF),
  \]
  which is nef because $\mathcal E\otimes\mathcal O_C(p)$ is globally generated.
  Moreover, $(X,\Delta)$ is geometrically klt: over $\overline k$, the divisor $D$ splits into two disjoint smooth sections avoiding the singular locus of $X_{\overline k}$.
  In particular, all the hypotheses of Proposition~\ref{proposition-MF-geo-normal} except smoothness hold, whereas its conclusion fails.
\end{example}


\begin{thebibliography}{CWZ26a}

\bibitem[Art66]{Artin66Rat}
M.~Artin.
\newblock On isolated rational singularities of surfaces.
\newblock {\em Am. J. Math.}, 88:129--136, 1966.

\bibitem[BFSZ24]{BFSZ2404}
Fabio Bernasconi, Andrea Fanelli, Julia Schneider, and Susanna Zimmermann.
\newblock Explicit Sarkisov program for regular surfaces over arbitrary fields and applications, 2024.
\newblock \href{https://arxiv.org/abs/2404.03281v3}{arXiv:2404.03281v3}.

\bibitem[BM24]{Bernasconi-Martin-2024-Bounding}
Fabio Bernasconi and Gebhard Martin.
\newblock Bounding geometrically integral del {Pezzo} surfaces.
\newblock {\em Forum Math. Sigma}, 12:e81, 1--24, 2024.

\bibitem[BT22]{Bernasconi-Tanaka-22}
Fabio Bernasconi and Hiromu Tanaka.
\newblock On del {Pezzo} fibrations in positive characteristic.
\newblock {\em J. Inst. Math. Jussieu}, 21(1):197--239, 2022.

\bibitem[BT24]{BT2408}
Fabio Bernasconi and Hiromu Tanaka.
\newblock Geometry and arithmetic of geometrically integral regular del Pezzo surfaces, 2024.
\newblock To appear in {\em Algebra \& Number Theory}.
\newblock \href{https://arxiv.org/abs/2408.11378v3}{arXiv:2408.11378v3}.

\bibitem[CWZ26a]{CWZ26a}
Jingshan Chen, Chongning Wang, and Lei Zhang.
\newblock On canonical bundle formula for fibrations of curves with arithmetic genus one.
\newblock Forum of Mathematics, Sigma, 2026.

\bibitem[CWZ26b]{CWZ26b}
Jingshan Chen, Chongning Wang, and Lei Zhang.
\newblock Effective nonvanishing of {$K$}-trivial regular surfaces over imperfect fields.
\newblock Preprint, 2026.
\newblock \href{https://arxiv.org/abs/2610.09542}{arXiv:2610.09542}.

\bibitem[DW22]{Das-Waldron-2022}
Omprokash Das and Joe Waldron.
\newblock On the log minimal model program for threefolds over imperfect fields of characteristic {$p>5$}.
\newblock {\em J. Lond. Math. Soc. (2)}, 106(4):3895--3937, 2022.

\bibitem[Eke87]{Ekedahl87F}
Torsten Ekedahl.
\newblock Foliations and inseparable morphisms.
\newblock In {\em Algebraic geometry, {B}owdoin, 1985 ({B}runswick, {M}aine, 1985)}, volume~46 of {\em Proc. Sympos. Pure Math.}, pages 139--149. Amer. Math. Soc., Providence, RI, 1987.

\bibitem[EP23]{EP23}
Sho Ejiri and Zsolt Patakfalvi.
\newblock The {Demailly--Peternell--Schneider} conjecture is true in positive characteristic.
\newblock 2023.
\newblock \href{https://arxiv.org/abs/2305.02157}{arXiv:2305.02157}.

\bibitem[FS20]{Fanelli-Schroer-2020}
Andrea Fanelli and Stefan Schr\"{o}er.
\newblock Del {P}ezzo surfaces and {M}ori fiber spaces in positive characteristic.
\newblock {\em Trans. Amer. Math. Soc.}, 373(3):1775--1843, 2020.

\bibitem[Ful98]{fulton_intersection_theory}
William Fulton.
\newblock {\em Intersection theory}, volume~2 of {\em Ergebnisse der Mathematik und ihrer Grenzgebiete. 3. Folge. A Series of Modern Surveys in Mathematics [Results in Mathematics and Related Areas. 3rd Series. A Series of Modern Surveys in Mathematics]}.
\newblock Springer-Verlag, Berlin, second edition, 1998.

\bibitem[Har77]{Hartshorne-AG}
Robin Hartshorne.
\newblock {\em Algebraic geometry}, volume No. 52 of {\em Graduate Texts in Mathematics}.
\newblock Springer-Verlag, New York-Heidelberg, 1977.

\bibitem[HW02]{Hara-Watanabe-2002}
Nobuo Hara and Kei-Ichi Watanabe.
\newblock F-regular and {F}-pure rings vs. log terminal and log canonical singularities.
\newblock {\em J. Algebraic Geom.}, 11(2):363--392, 2002.

\bibitem[JW21]{ji_structure_2021}
Lena Ji and Joe Waldron.
\newblock Structure of geometrically non-reduced varieties.
\newblock {\em Trans. Amer. Math. Soc.}, 374(12):8333--8363, 2021.

\bibitem[Kol13]{Kollar-Sing}
J{\'a}nos Koll{\'a}r.
\newblock {\em Singularities of the minimal model program. {With} the collaboration of {S{\'a}ndor} {Kov{\'a}cs}}, volume 200 of {\em Camb. Tracts Math.}
\newblock Cambridge: Cambridge University Press, 2013.

\bibitem[Lip69]{Lipman69}
Joseph Lipman.
\newblock Rational singularities, with applications to algebraic surfaces and unique factorization.
\newblock {\em Inst. Hautes \'{E}tudes Sci. Publ. Math.}, (36):195--279, 1969.

\bibitem[Liu02]{Liu-AGAC}
Qing Liu.
\newblock {\em Algebraic geometry and arithmetic curves}, volume~6 of {\em Oxford Graduate Texts in Mathematics}.
\newblock Oxford University Press, Oxford, 2002.
\newblock Translated from the French by Reinie Ern\'{e}, Oxford Science Publications.

\bibitem[Liu25]{Liu2025}
Qing Liu.
\newblock Desingularization of double covers of regular surfaces, 2025.
\newblock \href{https://arxiv.org/abs/2504.16808v3}{arXiv:2504.16808v3}.

\bibitem[Mat89]{MatsumuraCRT}
Hideyuki Matsumura.
\newblock {\em Commutative ring theory}, volume~8 of {\em Cambridge Studies in Advanced Mathematics}.
\newblock Cambridge University Press, Cambridge, second edition, 1989.
\newblock Translated from the Japanese by M. Reid.

\bibitem[Mat23]{Matsumoto-2023-mu_p-alpha_p}
Yuya Matsumoto.
\newblock {$\mu_p$}- and {$\alpha_p$}-actions on {K}3 surfaces in characteristic {$p$}.
\newblock {\em J. Algebraic Geom.}, 32(2):271--322, 2023.

\bibitem[MP97]{MP97}
Yoichi Miyaoka and Thomas Peternell.
\newblock {\em Geometry of higher dimensional algebraic varieties}, volume~26 of {\em DMV Semin.}
\newblock Birkh\"auser, Basel, 1997.

\bibitem[MS03]{Mori-Saito-03Wild}
Shigefumi Mori and Natsuo Saito.
\newblock Fano threefolds with wild conic bundle structures.
\newblock {\em Proc. Japan Acad., Ser. A}, 79(6):111--114, 2003.

\bibitem[PW22]{Patakfalvi-Waldron-2022}
Zsolt Patakfalvi and Joe Waldron.
\newblock Singularities of general fibers and the {LMMP}.
\newblock {\em Amer. J. Math.}, 144(2):505--540, 2022.

\bibitem[Que71]{Queen-71-I}
Clifford~S. Queen.
\newblock Non-conservative function fields of genus one. {I}.
\newblock {\em Arch. Math. (Basel)}, 22:612--623, 1971.

\bibitem[RS78]{Rudakov-Shafarevich-1976}
A.~N. Rudakov and I.~R. Shafarevich.
\newblock Inseparable morphisms of algebraic surfaces.
\newblock {\em Math. USSR, Izv.}, 10:1205--1237, 1978.

\bibitem[Sch10]{schroer_fibration_2010}
Stefan Schr\"oer.
\newblock On fibrations whose geometric fibers are nonreduced.
\newblock {\em Nagoya Math. J.}, 200:35--57, 2010.

\bibitem[Sch22]{Schroer2211}
Stefan Schröer.
\newblock The structure of regular genus-one curves over imperfect fields, 2022.
\newblock \href{https://arxiv.org/abs/2211.04073}{arXiv:2211.04073}.

\bibitem[{Sta}25]{stacks-project}
The {Stacks {P}roject authors}.
\newblock The {S}tacks {P}roject.
\newblock \url{https://stacks.math.columbia.edu}, Accessed in 2025.

\bibitem[Tan18]{Tanaka-2018-MMP-excelent-surface}
Hiromu Tanaka.
\newblock Minimal model program for excellent surfaces.
\newblock {\em Ann. Inst. Fourier}, 68(1):345--376, 2018.

\bibitem[Tan20]{tanaka_abundance_2020}
Hiromu Tanaka.
\newblock Abundance theorem for surfaces over imperfect fields.
\newblock {\em Mathematische Zeitschrift}, 295(1):595--622, 2020.

\bibitem[Tan21]{tanaka_invariants_2021}
Hiromu Tanaka.
\newblock Invariants of algebraic varieties over imperfect fields.
\newblock {\em Tohoku Math. J. (2)}, 73(4):471--538, 2021.

\bibitem[Tan24]{Tanaka-24-Bound}
Hiromu Tanaka.
\newblock Boundedness of regular del {P}ezzo surfaces over imperfect fields.
\newblock {\em Manuscripta Math.}, 174(1-2):355--379, 2024.

\bibitem[Ume81]{Umezu-1981}
Yumiko Umezu.
\newblock On normal projective surfaces with trivial dualizing sheaf.
\newblock {\em Tokyo J. Math.}, 4:343--354, 1981.

\end{thebibliography}

\end{document}